\documentclass[final,times,3p,nopreprintline]{elsarticle} 
\pdfoutput=1

\usepackage{framed,multirow}

\usepackage{mathtools,bbm}
\usepackage{amsfonts}
\usepackage{amssymb,latexsym}
\usepackage{amsbsy}
\usepackage{subfig,float,subfloat}
\usepackage{amsthm,bm}
\usepackage{lscape}
\usepackage[normalem]{ulem}
\usepackage{commath}
\usepackage[labelfont=bf]{caption}
\usepackage{pifont}
\usepackage{algorithm}
\usepackage{algorithmic}
\usepackage{listings}
\usepackage{xargs}
\usepackage[colorinlistoftodos,prependcaption,textsize=tiny]{todonotes}
\usepackage{multirow} 
\usepackage{tikz}
\usetikzlibrary{shapes.geometric}
\usetikzlibrary{arrows.meta,calc,decorations.pathreplacing,intersections}
\usepackage{tabulary}
\usepackage{enumitem}
\usepackage{centernot}
\usepackage[pagewise,displaymath, mathlines]{lineno}
\usepackage[linkcolor=blue,colorlinks=true]{hyperref}
\usepackage{cleveref}
\crefname{theorem}{Theorem}{Theorems}
\crefformat{theorem}{Theorem~#2#1#3}
\crefname{remark}{Remark}{Remarks}
\crefname{lemma}{Lemma}{Lemmas}
\crefname{table}{Table}{Tables}
\crefrangeformat{table}{Tables~#3#1#4 to~#5#2#6}
\crefname{figure}{Figure}{Figures}
\crefrangeformat{figure}{Figures~#3#1#4 to~#5#2#6}
\crefformat{equation}{#2#1#3}
\crefformat{appendix}{Appendix~#2#1#3}

\usepackage{natbib}
\usepackage{comment}

\newcommand{\x}{\normalfont{\textbf{x}}}

\newtheorem{remark}{\textbf{Remark}}[section]
\newtheorem{prop}{\textbf{Proposition}}[section]
\crefname{prop}{proposition}{propositions}
\Crefname{prop}{Proposition}{Propositions}
\newtheorem{assumption}{\textbf{Assumption}}[section]
\crefname{assumption}{assumption}{assumption}
\Crefname{assumption}{Assumption}{Assumption}
\newtheorem{corollary}{\textbf{Corollary}}[section]
\newtheorem{theorem}{\textbf{Theorem}}[section]
\newtheorem*{leib*}{\textit{Leibniz's rule}}
\newtheorem{lemma}{\textbf{Lemma}}[section]
\newtheorem{definition}{Definition}[section]
\crefname{prop}{definition}{definitions}
\Crefname{prop}{Definition}{Definitions}

\makeatletter
\newsavebox\myboxA
\newsavebox\myboxB
\newlength\mylenA
\newcommand*\xoverline[2][1.0]{%
    \sbox{\myboxA}{$\m@th#2$}%
    \setbox\myboxB\null
    \ht\myboxB=\ht\myboxA%
    \dp\myboxB=\dp\myboxA%
    \wd\myboxB=#1\wd\myboxA
    \sbox\myboxB{$\m@th\overline{\copy\myboxB}$}
    \setlength\mylenA{\the\wd\myboxA}
    \addtolength\mylenA{-\the\wd\myboxB}%
    \ifdim\wd\myboxB<\wd\myboxA%
       \rlap{\hskip 0.5\mylenA\usebox\myboxB}{\footnotesize{\usebox\myboxA}}%
    \else
        \hskip -0.5\mylenA\rlap{\usebox\myboxA}{\footnotesize{\hskip 0.5\mylenA\usebox\myboxB}}%
    \fi}
\makeatother

\newcommandx{\unsure}[2][1=]{\todo[linecolor=red,backgroundcolor=red!25,bordercolor=red,#1]{#2}}
\newcommandx{\change}[2][1=]{\todo[linecolor=blue,backgroundcolor=blue!25,bordercolor=blue,#1]{#2}}
\newcommandx{\info}[2][1=]{\todo[linecolor=OliveGreen,backgroundcolor=OliveGreen!25,bordercolor=OliveGreen,#1]{#2}}
\newcommandx{\improvement}[2][1=]{\todo[linecolor=Plum,backgroundcolor=Plum!25,bordercolor=Plum,#1]{#2}}
\newcommandx{\thiswillnotshow}[2][1=]{\todo[disable,#1]{#2}}

\usepackage{scalerel,stackengine}
\stackMath
\newcommand\widhat[1]{%
\savestack{\tmpbox}{\stretchto{%
  \scaleto{%
    \scalerel*[\widthof{\ensuremath{#1}}]{\kern-.6pt\bigwedge\kern-.6pt}%
    {\rule[-\textheight/2]{1ex}{\textheight}}
  }{\textheight}%
}{0.5ex}}%
\stackon[1pt]{#1}{\tmpbox}%
}
\newcommand{\R}{\mathbb{R}}
\newcommand{\E}{\mathbb{E}}

\usepackage{url}
\usepackage{xcolor}
\definecolor{newcolor}{rgb}{.8,.349,.1}

\title{Frictional martingale optimal transport and robust hedging}

\author[1,2]{Pratik Rai\corref{cor1}} 
\cortext[cor1]{\ead{pratik.rai@asnbank.nl}}

\address[1]{Quantitative Risk Management, ASN Bank}
\address[2]{Dept. of Mathematics and Computer Science, Eindhoven University of Technology}

\providecommand{\keywords}[1]
{
  \small	
  \textbf{Keywords. } #1
}

\providecommand{\AMS}[1]
{
  \small	
  \textbf{AMS subject classifications. } #1
}

\date{\today}

\begin{document}

\maketitle

\begin{abstract}
We study the martingale optimal transport problem with state‐dependent trading frictions and develop a geometric and duality framework extending from the one time–step to the multi–marginal setting. Building on the left–monotone structure of frictionless MOT (Beiglb{\"o}ck and Juillet, Ann. Probab., 2016; Henry-Labord{\`e}re and Touzi, Finance Stoch., 2016; Beiglb{\"o}ck et al., Ann. Probab., 2017), we introduce a convex frictional cost combining proportional bid–ask spreads and quadratic liquidity impacts. The framework extends the martingale Spence–Mirrlees condition to nonlinear frictions and establishes a frictional monotonicity principle.

At each time step, the joint distribution between consecutive asset prices exhibits a bi–atomic, monotone geometry: conditional on the current price, the next price lies on one of two monotone branches representing upward and downward rebalancing. A no–transaction region, or trade band, arises where maintaining the position is optimal, while outside the band, transitions follow two monotone graphs whose endpoints satisfy an equal–slope condition balancing continuation value and marginal trading cost.

The framework extends dynamically via a recursive identity, ensuring stability and convergence to the frictionless left–curtain limit, and applies to model‐independent pricing and robust hedging of path‐dependent derivatives.
\end{abstract}


\AMS{60G42, 91G20, 91G80, 49K30, 90C46}

\tableofcontents

\section{Introduction}
\label{sec:intro}

In robust hedging one seeks strategies and price bounds that do not rely on a single probabilistic model but are consistent with primitive market information such as the one–dimensional marginals of the underlying at a finite set of dates. From the viewpoint of an electronic market maker or a high–frequency trading desk, the central operational object is the inventory, i.e., the number of shares (or contracts) held between decision times $t$ and $t+1$, and rebalancing consists of executing an order of size $\Delta_t$ to move the position from $t$ to $t+1$. These trades are discrete and incur frictions: a spread–like linear cost proportional to $\abs{\Delta_t}$ and a quadratic or transient impact cost proportional to $\Delta^2_t$. Such costs directly shape turnover and the geometry of optimal rebalancing, echoing classical execution and market–making models in which proportional spreads and quadratic penalties proxy liquidity and impact \cite{AlmgrenChriss2001,AvellanedaStoikov2008,ObizhaevaWang2013,GueantLehalle2015}. In this setting, superhedging becomes especially relevant because attainable bounds depend not only on the payoff and the marginal constraints but also on the cost of transporting probability mass across states: the frictional penalty induces explicit inaction thresholds (no–trade bands) that identify when the inventory should be left unchanged, and—once these thresholds are crossed—determines target adjustments and order sizes for the next rebalance. Concretely, the linear component governs the width of the no–trade region (where small moves are optimally ignored), while the quadratic component calibrates the magnitude of the optimal jump in holdings when a trade is undertaken. As a result, the pricing problem yields operational prescriptions both for when not to trade and for how far to rebalance when one does.

Mathematically, the robust superhedging problem can be formulated as a stochastic control problem with a martingale state process subject to pathwise constraints. The objective is to maximize the expected payoff~$\Phi$ of a contingent claim over all martingale measures~$\pi$ whose marginals $(\mu_0,\mu_1,\dots,\mu_N)$ coincide with the market-implied distributions of the underlying asset at discrete trading dates. Each marginal~$\mu_t$ represents the \emph{risk-neutral law} of the underlying price at time~$t$, recoverable in practice from observed call prices via the Breeden--Litzenberger identity~\cite{BL1978}. In this setting, the superhedging price is the smallest initial capital required to dominate the payoff under all admissible martingale measures---equivalently, the largest value achievable by optimizing~$\Phi$ over all such models consistent with the observed marginals. This reformulation casts the problem as a martingale optimal transport (MOT) problem, where one seeks a joint law on~$\mathbb{R}^{N+1}$ with prescribed one-dimensional marginals that maximizes an expected payoff subject to the martingale constraint $\mathbb{E}[S_{t+1} \mid S_t] = S_t$~\cite{BLP2013,DolinskySoner2014,BeiglbockNutzTouzi2017}.  

The dual formulation of this problem admits a natural financial interpretation. It describes the construction of semi--static hedging portfolios combining predictable dynamic strategies in the underlying with static holdings in vanilla options, ensuring that the portfolio’s terminal value dominates the claim across all consistent models. The dual variables---typically a pair of functions $(\varphi_t, \psi_t)$ and a predictable slope~$h_t$---represent, respectively, the costs of static positions and the incremental trading strategy that enforces the martingale condition. Convex order emerges as the feasibility condition guaranteeing the existence of such martingale measures, while the dual problem expresses the sharpest model-independent price bounds attainable given the marginal information. In continuous time, this construction links directly to the Skorokhod embedding problem, which represents admissible martingale measures as embeddings of Brownian motion with a prescribed terminal distribution~\cite{BHR2001,O2004,H2011}. This connection has yielded robust bounds for a wide range of path-dependent claims, including barrier, lookback, and Asian options, as well as no-arbitrage conditions for implied volatility surfaces~\cite{C2004,C2007,DH2007,CHO2008,CO20111,CO20112,CW2013,OdRG2015,DOR2014,HN2012}. In discrete time, the martingale optimal transport approach provides an analogous structure as it characterizes extremal martingale couplings through \emph{left--monotone geometry}, guarantees the existence and uniqueness of optimizers, and gives explicit expressions for the superhedging bounds~\cite{BJ2016,HLT2016,BeiglbockNutzTouzi2017}.  

The present work builds on this discrete-time MOT foundation and extends it to incorporate state-dependent frictions, thereby enriching the geometry of optimal couplings. When transaction costs or execution frictions are present, the stochastic control problem is modified by an additional penalty on the displacement of the underlying across periods i.e. a penalty that captures the cost of transporting probability mass in state space. The resulting dual variables not only determine the hedging instruments and trading rates but also encode optimal rebalancing regions and no-trade zones, where marginal costs of adjustment exceed marginal benefits. This extension links robust hedging to modern models of optimal execution and liquidity provision, where the geometry of the optimal coupling translates directly into practical trading prescriptions for when to remain inactive, when to rebalance, and how the scale of the transaction should adjust to the frictional regime.

In this work we focus on the discrete–time setting, where the MOT formulation admits a particularly transparent structure on the real line. At the level of a single time–step, optimizers exhibit a left–monotone support and a bi–atomic disintegration on the active set where transport is required \cite{BJ2016,HLT2016}. On the line, dual attainment and full duality are by now well understood \cite{BeiglbockNutzTouzi2017}. Subsequent contributions have extended this framework to incorporate additional features and to refine regularity and stability properties. Introducing trading frictions through a convex penalty on the trade increment modifies the one time–step objective by subtracting a convex function of the velocity while preserving the martingale constraint, thereby changing the variational selection of the support. Linear–quadratic specifications, motivated by spread–dominated costs and transient impact, are canonical in the execution literature \cite{AlmgrenChriss2001,ObizhaevaWang2013} and integrate seemlessly wiithin the analytic structure of MOT. On the computational side, martingale analogues of the Benamou–Brenier dynamic formulation lead to convex saddle–point problems amenable to augmented–Lagrangian or ADMM schemes, providing practical algorithms for calibration and transport under martingale constraints \cite{GuoLoeperWang2018,GuoLoeperAAP2021}.

The central object of study here is the transport map generated by the frictional MOT problem in one dimension. In classical optimal transport with quadratic cost, Brenier’s theorem ensures that the optimal plan is induced by a single monotone map $T=\nabla\varphi$ pushing a measure $\mu$ to $\eta$ \cite{Brenier1991,McCann1995,Villani2009}. This static description aligns with the dynamical Benamou–Brenier formulation, where the geodesic $(\rho_t)_{t\in[0,1]}$ solves a continuity equation with minimal kinetic energy and initial velocity given by the Brenier vector field \cite{BenamouBrenier2000,Villani2009}. For general costs, the Spence–Mirrlees (or twist) condition $\partial_{xy}^2 c(x,y)<0$ rules out multi–valued correspondences and enforces a graph solution \cite{Carlier2003,Villani2009}. In martingale optimal transport the geometry is fundamentally different: the martingale constraint fixes conditional means, so even without frictions the optimizer typically splits mass along two monotone graphs leading to the left–curtain structure of \cite{BJ2016,HLT2016}. The appropriate one–dimensional monotonicity condition is a martingale Spence–Mirrlees property for the cost function $\widetilde c(x,y)=\mathsf V(y)-\mathsf V(x)-f(x,y-x)$. In a smooth regime this condition reads $\partial_{xyy}\widetilde c>0$ (equivalently, a rectangle inequality) and replaces the classical twist \cite{HLT2016,BeiglbockNutzTouzi2017}. With frictions, a new feature emerges in the form of no–trade region (or trade band) on which the identity coupling is optimal and mass remains on the diagonal. The band is characterized in dual variables by the subgradient condition $h_t(x)\in\partial_v f_t(x,0)$, while off the band the martingale kernel remains bi–atomic with endpoints selected by an equal–slope condition balancing continuation values against marginal trading costs. This geometry yields immediate implications for robust pricing of path–dependent claims and execution statistics such as turnover and cost, and it delivers comparative statics in liquidity parameters together with stability and vanishing–friction limits that recover the left–curtain coupling \cite{BJ2016,HLT2016,BeiglbockNutzTouzi2017,GuoLoeperAAP2021,GuoLoeperWang2018}.

The plan of the present work is as follows. Section~\ref{sec:setup} introduces the market setup, admissible payoffs, marginal constraints, and the frictional cost specification, and formulates the primal and dual problems. Section~\ref{sec:fric_geometry} develops the one time–step geometry of the frictional martingale optimal transport problem, establishing left–monotone structure, the equal–slope characterization, and the emergence of trade bands. Section~\ref{sec:bands_applications} analyzes trade–band geometry in detail, derives explicit displacement rules in the linear–quadratic case, and discusses their financial implications. Section~\ref{sec:multi_marginal} extends the analysis to the multi–marginal setting via dynamic programming, proving timewise preservation of the geometric structure and stability of optimal couplings. Section~\ref{sec:proof_strong_duality} establishes the proof for strong duality for the multi–marginal problem with state–dependent frictions.

\subsection{Contributions}

The main contributions of this paper are threefold.

\begin{enumerate}
    \item \textbf{Frictional one time–step geometry and trade bands.}  
    We introduce a frictional extension of the martingale optimal transport problem in discrete time, where transaction costs depend on both state and displacement through a convex function \( f^{(a,b)}_t(x,v) \).  
    Under a frictional martingale Spence–Mirrlees (MSM) condition, we establish a \emph{frictional monotonicity principle} showing that any optimizer admits a bi–atomic disintegration supported on two monotone graphs \( T_d^{(t)} \) and \( T_u^{(t)} \), together with an identity band \( B_t = \{x : h_t(x) \in \partial_v f^{(a,b)}_t(x,0)\} \) on which the no–trade condition holds.  
    The band boundaries are characterized by an equal–slope system balancing continuation values and marginal trading costs, thereby connecting convex analysis, geometry, and trading behavior.

    \item \textbf{Stability and vanishing–friction limit.}  
    We prove that the family of optimal couplings and endpoints is stable under joint perturbations of marginals and friction parameters.  
    In particular, when the linear and quadratic friction coefficients \((\alpha_n,\beta_n)\) vanish, the frictional optimizer converges in law and \(L^1\) to the frictionless left–curtain coupling of \cite{BJ2016,HLT2016}.  
    This provides a quantitative bridge between frictional and frictionless martingale geometries and clarifies how transaction costs regularize the support of optimal transport.

    \item \textbf{Multi–marginal extension and strong duality.}  
    Using dynamic programming and measurable selection, we extend the one–step analysis to a multi–marginal setting, showing that the bi–atomic structure and trade–band geometry propagate across time.  
    We then establish strong duality between the primal martingale optimal transport problem with frictions and its dual formulation involving semi–static portfolios.

\end{enumerate}

Beyond their theoretical interest, these results provide a rigorous foundation for understanding the interaction between liquidity costs, hedging geometry, and robust pricing. The framework developed here unifies concepts from optimal transport, stochastic control, and execution theory, and opens the way to practical algorithms for frictional calibration and stability analysis in financial markets.

\section{Setup and Problem Formulation}
\label{sec:setup}

\subsection{Discrete-time market model and marginal structure}
\label{subsec:pathspace_marginals}

We consider a discrete-time financial market model defined over a finite time horizon \( T > 0 \), subdivided into \( N \in \mathbb{N} \) periods. The canonical path space is given by
\[
S_t(\omega) := \omega_t, \quad \text{for } t = 0,\dots,N,
\]
where each element \( \omega = (\omega_0,\dots,\omega_N) \in \mathbb{R}^{N+1} \) describes a discrete trajectory of the price process. The coordinate process \( (S_t)_{t=0}^N \) represents the underlying asset price at each time step, and is endowed with the natural filtration $\mathcal{F}_t := \sigma(S_0,\dots,S_t)$, capturing the available information up to time $t$. We denote by $\mathcal{F} := \mathcal{F}_N$ the terminal sigma-algebra.

The pricing framework is model-independent in the sense that we do not specify a single probability measure to govern the underlying asset dynamics. Instead, we are only given a sequence of one-dimensional marginal distributions $(\mu_t)_{t=0}^N \subset \mathcal{P}(\mathbb{R})$, prescribing the law of $S_t$ at each time $t$. 

Here we recall the assumptions that we place on the marginals:

\begin{assumption}[Convex order and peacock]
\label{ass:convex_order}

\smallskip
\begin{itemize}
  \item[\bf A1.] For every \( t = 0, \dots, N-1 \), we assume that
  \[
  \mu_t \preceq_{\mathrm{cx}} \mu_{t+1},
  \]
  i.e., the sequence \( (\mu_t) \) is increasing in convex order. Moreover, there exists \( p \geq 2 \) such that
  \[
  \int_{\mathbb{R}} |x|^p\,\mathrm{d}\mu_t(x) < \infty \quad \text{for all } t = 0,\dots,N.
  \]
  \item[\bf A2.] In particular, all \(\mu_t\) have the same mean, \(\int x\,\mathrm d\mu_t(x)=s_0\).
\end{itemize}
The convex ordering assumption ensures the existence of martingale couplings, while the moment condition guarantees sufficient integrability for evaluating the cost function.
\end{assumption}

Let
\begin{equation}
\label{eq:martingale_measures}
\mathcal{M}(\mu_0,\dots,\mu_N):=\Big\{\pi\in\mathcal{P}(\Omega): (S_t)_{t=0}^N \text{ is a }\pi\text{-martingale and } \mathcal{L}_\pi(S_t)=\mu_t\ \forall t\Big\}
\end{equation}
define the set of coupling measures. By Kellerer’s theorem \cite{kellerer1972markov}, \(\mathcal{M}(\mu_0,\dots,\mu_N)\neq\varnothing\) under Assumption~\ref{ass:convex_order}. This construction essentially reflects the scenario of model uncertainty where only the marginal distributions of the asset price are known, and the hedger seeks robustness over all consistent models.

\subsection{Admissible Payoff}
\label{subsec:admissible_payoff}

The terminal payoff of a derivative contract is modeled as a measurable function $\Phi : \Omega \to \mathbb{R} \cup \{+\infty\}$, depending on the entire trajectory of the underlying asset. Examples to such contracts include path-dependent options such as lookbacks, Asian options, and barrier contracts.

To ensure the well-posedness of the robust pricing problem and the validity of the duality arguments, we impose the following regularity and integrability conditions on $\Phi$.

\begin{definition}[Admissible Payoff]
\label{def:admissible_payoff}
A function $\Phi : \Omega \to \mathbb{R} \cup \{+\infty\}$ is called admissible if:
\begin{enumerate}
    \item[(i)] For every \( \pi \in \mathcal{M}(\mu_0,\dots,\mu_N) \), the negative part \( \Phi^- := \max\{-\Phi, 0\} \) is integrable, i.e.,
    \[
    \mathbb{E}_\pi[\Phi^-] < \infty;
    \]
    \item[(ii)] \(\Phi\) is upper semicontinuous on \(\Omega\) and has at most \(p\)-order growth:
    \[
    \exists C>0:\quad \Phi(\omega)\le C\Big(1+\sum_{t=0}^N|\omega_t|^p\Big)\qquad \text{for all }\omega\in\Omega.
    \]
\end{enumerate}
\end{definition}

\subsection{Trading Strategies and Frictional Cost}
\label{subsec:trading_friction}

We formalize the dynamic component of a semi–static hedging strategy in the presence of execution costs. Let \( \Delta=(\Delta_0,\dots,\Delta_{N-1}) \) be a predictable process, where each \( \Delta_t:\R^{t+1}\to\R \) is \( \mathcal F_t \)-measurable and represents the number of shares held on the interval \([t,t{+}1]\). These discrete–time controls are the analogue of delta strategies determined by the observed price history up to time \(t\).

To model trading frictions, we introduce for each \(t=0,\dots,N{-}1\) a Borel function
\begin{equation}
\label{eq:friction_function}
f^{(a,b)}_t(x,\delta)\;:=\;a_t(x)\,|\delta|+b_t(x)\,\delta^2,\qquad (x,\delta)\in\R\times\R,
\end{equation}
where \(a_t,b_t:\R\to[0,\infty)\) are given. The map \(\delta\mapsto f^{(a,b)}_t(x,\delta)\) is convex, continuous, and satisfies \(f^{(a,b)}_t(x,0)=0\). The linear term \(a_t(x)|\delta|\) captures proportional costs such as bid–ask spreads or fees, while the quadratic term \(b_t(x)\delta^2\) models transient price impact and inventory penalization, in line with the execution literature \cite{AlmgrenChriss2001,ObizhaevaWang2013,GueantLehalle2015,AvellanedaStoikov2008}. When uniqueness or strict convexity in the one time–step selection problem is required, we impose \(0<\underline b\le b_t(x)\le\overline b<\infty\) and \(0\le a_t(x)\le\overline a\).

This specification differs from frameworks that treat frictions purely as proportional transaction costs. Classical models of proportional costs are well developed in the literature on consistent price systems and no–arbitrage under transaction costs, see for instance \cite{schachermayer2004fundamental,kabanov2009markets}. In the discrete–time robust hedging context, \cite{dolinsky2013duality,DolinskySoner2014} build on this framework by considering binomial–type approximations with proportional, state–independent costs, and prove duality results in that setting. Our formulation extends beyond these models by allowing both nonlinear impact and state dependence through \(x\mapsto a_t(x),b_t(x)\), thereby accommodating heterogeneous liquidity conditions and varying execution penalties across states while remaining compatible with the convex–analytic martingale optimal transport approach used throughout this paper.

Here we also introduce the assumptions on the hybrid friction function \eqref{eq:friction_function} which are essential to deduce the result in \cref{thm:strong_duality}. These properties ensure that the frictional penalty dominates any potential arbitrage through excessive trading, thereby providing regularity and compactness in the optimization problem \eqref{eq:superhedging_price_revised}.

\begin{assumption}
\label{ass:fric-assump}
For each $t=0,\dots,N-1$ that
\begin{equation}\label{eq:fric-assump}
\begin{aligned}
&\text{(i) } v\mapsto f^{(a,b)}_t(x,v)\ \text{is proper, convex, lower semicontinuous, and } f^{(a,b)}_t(x,0)=0,\\
&\text{(ii) superlinearity:}\quad \lim_{|v|\to\infty}\frac{f^{(a,b)}_t(x,v)}{|v|}=+\infty
\ \text{ uniformly for $x$ in compact sets},\\
&\text{(iii) coercive growth:}\quad 
f^{(a,b)}_t(x,v)\ \ge\ m\,|v|^p - c\,(1+|x|)\quad\text{for some }m>0,\ p>1,\ c<\infty.
\end{aligned}
\end{equation}
Then for any $\pi\in\mathcal M(\mu_0,\dots,\mu_N)$,
\[
\sum_{t=0}^{N-1}\E_\pi \big[f^{(a,b)}_t(S_t,S_{t+1}-S_t)\big]\ \ge\ 
m\sum_{t=0}^{N-1}\E_\pi \big[|S_{t+1}-S_t|^p\big]-C,
\]
so the objective is coercive in the increments and uniformly integrable along maximizing sequences.
\end{assumption}

\begin{lemma}[Conjugate of \eqref{eq:friction_function}]\label{lem:conjugate_state}
For each $(t,x)$ and $y\in\R$,
\begin{equation}\label{eq:conjugate_friction}
f^{(a,b)\,*}_t(x,y)
=\sup_{\delta\in\R}\big\{y\delta-f^{(a,b)}_t(x,\delta)\big\}
=
\begin{cases}
\displaystyle\frac{\big(|y|-a_t(x)\big)_+^2}{4\,b_t(x)}, & b_t(x)>0,\\[8pt]
\iota_{[-a_t(x),\,a_t(x)]}(y), & b_t(x)=0,
\end{cases}
\end{equation}
where $\iota_A(y):=0$ if $y\in A$ and $+\infty$ otherwise.
\end{lemma}

\begin{proof}
Fix $t$ and $x\in\R$, and abbreviate $a:=a_t(x)\ge 0$ and $b:=b_t(x)\ge 0$. For $y\in\R$,
\[
f^{(a,b)\,*}_t(x,y)
=\sup_{\delta\in\R}\big\{y\delta-a|\delta|-b\delta^2\big\}.
\]

\emph{Case $b=0$.} Then $y\delta-a|\delta|\le (|y|-a)\,|\delta|$. If $|y|>a$, the right–hand side tends to $+\infty$ as $|\delta|\to\infty$, hence the supremum is $+\infty$. If $|y|\le a$, the expression is $\le 0$ for all $\delta$, with equality at $\delta=0$. Thus
\(
f^{(a,0)\,*}_t(x,y)=\iota_{[-a,a]}(y).
\)

\emph{Case $b>0$.} Consider $\phi(\delta):=y\delta-a|\delta|-b\delta^2$ and split the maximization over $\delta\ge 0$ and $\delta\le 0$.

For $\delta\ge 0$ we have $\phi(\delta)=(y-a)\delta-b\delta^2$, a concave quadratic on $[0,\infty)$ whose unconstrained maximizer is $\delta_+^\ast=(y-a)/(2b)$. Hence
\[
\sup_{\delta\ge 0}\phi(\delta)=
\begin{cases}
\dfrac{(y-a)^2}{4b}, & y\ge a,\\[6pt]
0, & y<a,
\end{cases}
\]
where the second line is attained at the boundary $\delta=0$.

For $\delta\le 0$ we have $|\delta|=-\delta$ and $\phi(\delta)=(y+a)\delta-b\delta^2$, a concave quadratic on $(-\infty,0]$ whose unconstrained maximizer is $\delta_-^\ast=(y+a)/(2b)$. Therefore
\[
\sup_{\delta\le 0}\phi(\delta)=
\begin{cases}
\dfrac{(y+a)^2}{4b}, & y\le -a,\\[6pt]
0, & y>-a,
\end{cases}
\]
again with the second line attained at $\delta=0$.

Taking the maximum of the two one–sided suprema yields
\[
f^{(a,b)\,*}_t(x,y)=
\begin{cases}
\dfrac{(y-a)^2}{4b}, & y\ge a,\\[6pt]
0, & |y|<a,\\[6pt]
\dfrac{(y+a)^2}{4b}, & y\le -a,
\end{cases}
\]
which can be written equivalently in the compact form
\(
f^{(a,b)\,*}_t(x,y)=\big(|y|-a\big)_+^2/(4b).
\)
\end{proof}

\subsection{Portfolio Value and Frictional Superhedging Price}
\label{subsec:portfolio_superhedging}

In an incomplete market with transaction costs, an investor may hedge a payoff $\Phi$ using a combination of static and dynamic trading strategies. The static part consists of positions in a set of European options written on the terminal asset price \( S_N \), while the dynamic part comprises of predictable rebalancing strategies based on the observed path of the underlying up to each trading date.

Let \( \lambda : \mathbb{R} \to \mathbb{R} \) be a measurable function and \( \lambda(S_N) \) be the European claim, then the initial cost of acquiring this static position is given by
\begin{equation}
\label{eq:static_cost_revised}
\nu(\lambda) := \int_{\mathbb{R}} \lambda(x)\,\mathrm{d}\mu_N(x).
\end{equation}

In addition to the static hedge, the investor deploys a dynamic trading strategy \( \Delta = (\Delta_0,\dots,\Delta_{N-1}) \) as described in Section~\ref{subsec:trading_friction}. The initial capital \( Y_0 \in \mathbb{R} \) is allocated between financing the static hedge and covering the dynamic trading path. The cumulative value of the strategy at terminal time \( N \) includes gains from trading, the payoff from the static option position, and the price of market friction. It is given by
\begin{equation}
\label{eq:terminal_wealth_revised}
Y_N^{\mathrm{fric}} = Y_0 + \sum_{t=0}^{N-1} \Delta_t(S_0,\dots,S_t)(S_{t+1} - S_t)  
+ \lambda(S_N) - \nu(\lambda) -\sum_{t=0}^{N-1} f^{(a,b)*}_t~\big(S_t,\Delta_t(S_0,\dots,S_t)\big).
\end{equation}

Here we make precise the modeling choice behind \eqref{eq:terminal_wealth_revised}. Let \(\pi\in\mathcal M(\mu_0,\dots,\mu_N)\) be a feasible martingale law on \((S_0,\dots,S_N)\), and denote by
\begin{equation}\label{eq:pi_t, V_t}
\pi_t\ :=\ (S_t,S_{t+1})_\#\pi\ \in\ \mathcal P(\R^2),\qquad V_t:=S_{t+1}-S_t .
\end{equation}
In the increment–penalty specification used in \eqref{eq:primal_problem_revised}, the model-side objective incorporates the friction cost through
\begin{equation}\label{eq:inc-penalty-primal}
\sum_{t=0}^{N-1}\E_\pi\big[f^{(a,b)}_t(S_t,V_t)\big]
\ =\ \sum_{t=0}^{N-1}\int_{\R^2} f^{(a,b)}_t(x,\,y-x)\,\mathrm d\pi_t(x,y) =: \mathcal{C}(\pi).
\end{equation}
On the hedger side, for a predictable slope process \(h=(h_t)_{t=0}^{N-1}\) with \(h_t=h_t(S_0,\dots,S_t)\), the Fenchel–Young inequality applied pointwise yields, for every \(t\),
\begin{equation}\label{eq:fenchel-young-pathwise}
-\,f^{(a,b)}_t\big(S_t,V_t\big)\ \le\ h_t(S_t)\,V_t\ -\ f^{(a,b)\,*}_t\big(S_t,\,h_t(S_t)\big),
\end{equation}
where \(f^{(a,b)\,*}_t\) is the convex conjugate in the second variable. Summing \eqref{eq:fenchel-young-pathwise} over \(t\) and taking expectation under any \(\pi\in\mathcal M(\mu_0,\dots,\mu_N)\) gives the dual superhedging inequality
\begin{equation}\label{eq:pathwise-dual-inequality}
\E_\pi \left[\Phi(S_0,\dots,S_N)\ -\ \sum_{t=0}^{N-1} f^{(a,b)}_t(S_t,V_t)\right]
\ \le\ 
\E_\pi \left[\ \Phi(S_0,\dots,S_N)\ +\ \sum_{t=0}^{N-1}\Big(h_t(S_t)\,V_t\ -\ f^{(a,b)\,*}_t\big(S_t,h_t(S_t)\big)\Big)\ \right],
\end{equation}
which is the algebraic form underlying \eqref{eq:terminal_wealth_revised} and the dual problem (the right-hand side represents terminal wealth of a predictable strategy with penalty \(f^{*}\) paid pathwise).

For comparison, an execution–cost specification charges directly for changes in the position process \(\Delta=(\Delta_t)_{t=0}^{N}\). Writing the trade size as \(\delta_t:=\Delta_{t+1}-\Delta_t\), the terminal wealth of a self-financing strategy with execution costs is
\begin{equation}\label{eq:exec-wealth}
X_N\ =\ X_0\ +\ \sum_{t=0}^{N-1}\Delta_t\,(S_{t+1}-S_t)\ -\ \sum_{t=0}^{N-1} f^{(a,b)}_t\big(S_t,\delta_t\big).
\end{equation}
In this formulation the dual constraint arises by maximizing \(\delta\mapsto h_t\,\delta - f^{(a,b)}_t(S_t,\delta)\), i.e.
\begin{equation}\label{eq:exec-dual-block}
\sup_{\delta\in\R}\,\{\,h_t\,\delta - f^{(a,b)}_t(S_t,\delta)\,\}\ =\ f^{(a,b)\,*}_t\big(S_t,h_t\big),
\end{equation}
so that the dual penalty involves the same conjugate \(f^{*}\) but now as the Legendre transform in the trading variable \(\delta\). We adopt the increment–penalty model \eqref{eq:inc-penalty-primal} because it couples directly to the martingale optimal transport objective \eqref{eq:primal_problem_revised} through \eqref{eq:fenchel-young-pathwise}–\eqref{eq:pathwise-dual-inequality}, yielding a clean pathwise dual representation.

\begin{definition}[Frictional Superhedging Price]
\label{def:frictional_superhedging}
Let $\Phi$ be the payoff function that satisfies \cref{def:admissible_payoff}. The frictional superhedging price of \( \Phi \) is defined as
\begin{equation}
\label{eq:superhedging_price_revised}
\mathsf{SH}(\Phi) := \inf \left\{ Y_0 \in \mathbb{R} : \exists\, \lambda,\, \Delta \text{ s.t. } 
Y_N^{\mathrm{fric}} \geq \Phi \quad \text{$\pi$-a.s. for all } \pi \in \mathcal{M}(\mu_0,\dots,\mu_N) \right\}.
\end{equation}
\end{definition}

Therefore, \eqref{eq:superhedging_price_revised} represents the minimal capital $Y_N^{\mathrm{fric}}$ required to hedge the claim \( \Phi \) in a model-independent way, while accounting for all possible price evolutions compatible with the prescribed marginals and respecting the martingale property. Here the inclusion of frictions ensures that the replication cost reflects real-world trading constraints and precludes over-optimistic hedging strategies that rely on idealized, costless markets.

\subsection{Primal problem}
\label{subsec:primal_transport}

The superhedging problem introduced in Definition~\ref{def:frictional_superhedging} admits an elegant reformulation as a penalized martingale optimal transport (MOT) problem.

Let \( \pi \in \mathcal{M}(\mu_0,\dots,\mu_N) \) denote a discrete-time martingale measure on \( \mathbb{R}^{N+1} \), and let \( \pi_t \), defined in \eqref{eq:pi_t, V_t}, be the marginal law of \( (S_t, S_{t+1}) \) under \( \pi \). The total trading cost associated with a given transport plan \( \pi \) is defined by $\mathcal{C}(\pi)$ as in \eqref{eq:inc-penalty-primal}. Consequently, in \eqref{eq:inc-penalty-primal}, each term $f^{(a,b)}_t(x, y - x)$ encodes the cost incurred when trading from a position at $x$ to $y$, and the integrals over $\pi_t$ aggregate these costs across all possible transitions allowed by the transport plan.

The primal formulation of the robust pricing problem then becomes
\begin{equation}
\label{eq:primal_problem_revised}
\mathbf{V} := \sup_{\pi \in \mathcal{M}(\mu_0,\dots,\mu_N)} \left\{ \mathbb{E}_\pi[\Phi] - \mathcal{C}(\pi) \right\}.
\end{equation}

This variational representation expresses the superhedging value as the maximal expected payoff over all martingale couplings, penalized by the cumulative cost of trading. The term $\mathbb{E}_\pi[\Phi]$ reflects the potential gain from a successful hedge under the worst-case path scenario, while $\mathcal{C}(\pi)$ imposes a penalty for incurring liquidity or execution costs.

Compared to the classical MOT framework studied in~\cite{BLP2013, DolinskySoner2014} which assumes frictionless markets, the inclusion of $ \mathcal{C}(\pi) $ introduces a regularization that penalizes rapid or volatile movements in the price path. This distinguishes our approach from the frictionless setting and aligns it more closely with the real-world trading environment. 

\subsection{Dual Formulation: From Portfolio Inequality to Strong Duality}
\label{subsec:dual_formulation}

We now derive the dual formulation of the superhedging problem under frictions. Recall the expression for the terminal wealth in \eqref{eq:terminal_wealth_revised} and the superhedging condition
\begin{equation}\label{eq:dual_superhedge_condition}
Y_N^{\mathrm{fric}}(\omega) \;\ge\; \Phi(\omega) \qquad \text{for all } \omega \in \Omega.
\end{equation}
Using the Fenchel--Young inequality, we note that for each friction function \( f^{(a,b)}_t \) and its Legendre--Fenchel conjugate \( f^{(a,b)*}_t \), it holds
\begin{equation}
f^{(a,b)}_t(x,\delta) + f^{(a,b)*}_t(x,y) \;\ge\; y \delta,
\end{equation}
with equality if and only if \( y \in \partial_\delta f^{(a,b)}_t(x,\delta) \). Applying this inequality pathwise with \(\delta=S_{t+1}-S_t\), we obtain:
\begin{equation}\label{eq:pathwise_dual_constraint}
\Delta_t(S_{t+1}-S_t)-f^{(a,b)*}_t(S_t,\Delta_t)\ \le\ f^{(a,b)}_t\big(S_t,S_{t+1}-S_t\big).
\end{equation}

We now encode the dynamic trading part via the frictional cost and the static position via a potential. Set \( u_N = \lambda \) and \( u_t \equiv 0 \) for \( t < N \), and define the pathwise dual functional
\begin{equation}\label{eq:Psi_dual_functional}
\Psi_{u,\Delta}(\omega):=\sum_{t=0}^N u_t\big(S_t(\omega)\big)
+\sum_{t=0}^{N-1}\Big[\Delta_t(\omega)\,(S_{t+1}-S_t)(\omega)-f^{(a,b)*}_t\big(S_t(\omega),\Delta_t(\omega)\big)\Big].
\end{equation}
Then the superhedging constraint becomes equivalent to
\begin{equation}
\label{eq:dual_inequality}
\Phi(\omega) \;\le\; \Psi_{u,\Delta}(\omega) + Y_0 - \nu(\lambda), \qquad \text{for all } \omega \in \Omega.
\end{equation}
Minimizing the initial capital \( Y_0 \) subject to this inequality corresponds to minimizing \( \sum_{t=0}^{N} \int u_t \, \mathrm{d}\mu_t \) over all admissible dual variables.

We thus define the dual domain:
\begin{equation}\label{eq:dual_admissible_set}
\mathcal{D}:=\Big\{(u,\Delta):\ u_t\in L^1(\mu_t),\ \Delta\ \text{predictable},\ \Phi\le \Psi_{u,\Delta}\ \text{pointwise on }\Omega\Big\}
\end{equation}
and the dual value is given by:
\begin{equation}\label{eq:dual_problem}
\mathbf{D} := \inf_{(u, \Delta) \in \mathcal{D}} \sum_{t=0}^N \int u_t \, \mathrm{d}\mu_t.
\end{equation}

\begin{theorem}[Strong Duality with State-Dependent Frictions]
\label{thm:strong_duality}
Suppose that the marginal constraints satisfy the convex ordering condition of Assumption~\ref{ass:convex_order}, the claim \( \Phi \) is admissible in the sense of Definition~\ref{def:admissible_payoff}, and the friction cost functions \( f^{(a,b)}_t \) satisfy the convexity and superlinearity assumptions. Then the primal and dual problems satisfy strong duality:
\begin{equation}\label{eq:strong_duality}
\mathbf{V} 
= \sup_{\pi \in \mathcal{M}(\mu_0,\dots,\mu_N)} 
\left\{ \mathbb{E}_\pi[\Phi] 
- \sum_{t=0}^{N-1} \int_{\mathbb{R}^2} f^{(a,b)}_t(x, y - x) \, \mathrm{d}\pi_t(x, y)
\right\}
= \mathbf{D}.
\end{equation}
Moreover, both extrema are attained.
\end{theorem}

See section~\ref{sec:proof_strong_duality} for the proof.

\begin{remark}[Subhedging with frictions]
\label{rem:subhedging}
In parallel with the superhedging problem, one may consider the \emph{subhedging} value of a contingent claim~$\Phi$, i.e.\ the maximal guaranteed initial capital that can be raised by selling $\Phi$ while trading dynamically under frictions. Formally, the subhedging problem is
\begin{equation}\label{eq:subhedging_primal}
\mathbf{V}_{\mathrm{sub}}
:=\inf_{\pi\in\mathcal M(\mu_0,\dots,\mu_N)}
\left\{ \E_\pi[\Phi]
+\sum_{t=0}^{N-1}\int_{\R^2} f^{(a,b)}_t(x,y-x)\,\mathrm d\pi_t(x,y)\right\}.
\end{equation}
This is the natural counterpart to \eqref{eq:primal_problem_revised}, with the frictional penalty added rather than subtracted. For the dual problem, the subhedging problem can be expressed as
\begin{equation}\label{eq:subhedging_dual}
\mathbf{D}_{\mathrm{sub}}
:=\sup_{(u,\Delta)}\ \sum_{t=0}^N\int u_t\,\mathrm d\mu_t,
\end{equation}
where the supremum is taken over pairs $(u,\Delta)$ such that
\[
\Phi\;\ge\;\Psi_{u,\Delta}\quad\text{pointwise on }\Omega,
\]
with $\Psi_{u,\Delta}$ defined in \eqref{eq:Psi_dual_functional}. As in the superhedging case, convexity and superlinearity of the cost functions $f^{(a,b)}_t$ ensure strong duality $\mathbf V_{\mathrm{sub}}=\mathbf D_{\mathrm{sub}}$, and both extrema are attained.
\end{remark}

\section{The one time–step geometry under friction}
\label{sec:fric_geometry}

Here we will develop the structure of the optimal martingale coupling for the marginals at time $t$ and $t+1$ in the presence of frictions \cite{BeiglbockLabordereTouzi2017, BJ2016, HLT2016}. The central tool is an \emph{adjusted} cost function that combines the continuation value of the payoff with the trading friction. 

Fix $t\in\{0,\dots,N-1\}$ and abbreviate $\mu:=\mu_t$, $\eta:=\mu_{t+1}$. Denote by $X(x,y):=x$ and $Y(x,y):=y$ the coordinate projections and set
\[
\Pi^M(\mu,\eta)
:=\Bigl\{\pi\in\mathcal{P}(\mathbb{R}^2):\ \pi\circ X^{-1}=\mu,\ \pi\circ Y^{-1}=\eta,\ \mathbb{E}_\pi[Y\mid X]=X\Bigr\}.
\]
Under \cref{ass:convex_order}, the continuation value $\mathsf V:\R\to\R$ at time $t{+}1$ is defined as
\begin{equation}
\label{eq:continuation_value_two_marg}
\mathsf V(y)
:= \sup_{\pi'\in\mathcal{M}(\mu_{t+1},\dots,\mu_N)}
\mathbb{E}_{\pi'} \left[\Phi(S_0,\dots,S_N)\ \Big|\ S_{t+1}=y\right],\qquad y\in\R,
\end{equation}
which depends only on the next marginal $\mu_{t+1}$ and the future marginals
$(\mu_{t+2},\dots,\mu_N)$.

The adjusted one time–step cost function is defined as
\begin{equation}
\label{eq:adjusted_cost_repeat}
\widetilde c_t(x,y):=\mathsf V(y)-\mathsf V(x)-f^{(a,b)}_t \big(x,\,y-x\big),
\qquad (x,y)\in\R^2,
\end{equation}
and consider
\begin{equation}
\label{eq:onestep_objective_adjusted}
\sup_{\pi\in\Pi^M(\mu,\eta)} \int_{\R^2} \widetilde c_t\,d\pi
\;=\; \eta(\mathsf V)-\mu(\mathsf V)\;-\;\inf_{\pi\in\Pi^M(\mu,\eta)} \int_{\R^2} f^{(a,b)}_t \big(x,\,y-x\big)\,d\pi.
\end{equation}
Thus the term $\mathsf V(y)-\mathsf V(x)$ contributes an additive constant
$\eta(\mathsf V)-\mu(\mathsf V)$ independent of $\pi$, and the argmax of the
adjusted problem coincides with the minimizers of the pure–friction functional
$\pi\mapsto\int f^{(a,b)}_t(x,y-x)\,d\pi$.

\subsection{Martingale Spence–Mirrlees (MSM) condition}
\label{subsec:fric_MSM}

The one time–step problem relies on a variational comparison over local martingale
competitors and a curvature condition on \(\widetilde c_t\).
This is the one–dimensional monotonicity mechanism of
\cite{HLT2016,BeiglbockLabordereTouzi2017}, adapted to the frictional setting, where the monotonicity is obtained by testing optimality against elementary rectangle
perturbations among martingale competitors.

\begin{definition}[Martingale competitors]
\label{def:competitors}
Let \(\alpha\) be a finite Borel measure on \(\R^2\) with finite first moments and
\(x\)–marginal \(\mu_\alpha:=\alpha\circ X^{-1}\).
A finite Borel measure \(\alpha'\) on \(\R^2\) is a martingale competitor of \(\alpha\) if
\[
\alpha'\circ X^{-1}=\mu_\alpha,\qquad \alpha'\circ Y^{-1}=\alpha\circ Y^{-1},
\]
and, for \(\mu_\alpha\)–a.e.\ \(x\),
\[
\int y\,\alpha(\mathrm dy\mid x)=\int y\,\alpha'(\mathrm dy\mid x).
\]
Equivalently, \(\E_{\alpha'}[Y\mid X]=\E_{\alpha}[Y\mid X]\) holds \(\mu_\alpha\)–a.e.
\end{definition}

\begin{definition}[Rectangle MSM inequality]
\label{def:rectangle_MSM}
We say that \(\widetilde{c}_t\) satisfies the martingale Spence–Mirrlees condition if,
for all \(x<x'\) and \(y^-<y^+\),
\begin{equation}
\label{eq:rectangle_MSM}
\Delta^{\square}\widetilde{c}_t\big((x,x');(y^-,y^+)\big)
:=\widetilde{c}_t(x,y^-)+\widetilde{c}_t(x',y^+)-\widetilde{c}_t(x,y^+)-\widetilde{c}_t(x',y^-)\ >\ 0.
\end{equation}
When \(\widetilde{c}_t\in C^3\), a sufficient pointwise criterion is
\begin{equation}
\label{eq:MSM_smooth}
\partial_{xyy}\widetilde{c}_t(x,y)\ \ge\ \kappa_t\ >\ 0\qquad \text{for all }(x,y)\in\R^2,
\end{equation}
see \cite[Sec.~3]{HLT2016}, \cite{BJ2016} and \cite[Thm.~5.2]{BeiglbockNutzTouzi2017}.
\end{definition}

Recall that
\[
\widetilde{c}_t(x,y)
=\mathsf{V}_{t+1}(y)-\mathsf{V}_{t+1}(x)-f^{(a,b)}_t(x,y-x),
\]
and write \(v:=y-x\). The separable part \(\mathsf V_{t+1}(y)-\mathsf V_{t+1}(x)\) has zero rectangle increment, and differentiating yields
\[
\partial_{xyy}\widetilde c_t(x,y)
=\underbrace{0}_{\text{from }\mathsf V_{t+1}(y)-\mathsf V_{t+1}(x)}
\;+\; (f^{(a,b)}_t)_{vvv}(x,v)\;-\;(f^{(a,b)}_t)_{xvv}(x,v).
\]
Thus the frictional MSM requirement (\eqref{eq:rectangle_MSM} or \eqref{eq:MSM_smooth}) amounts to a positivity condition on the \(v\)–curvature of \(f^{(a,b)}_t\) relative to its \(x\)–dependence. In the \(x\)–independent linear–quadratic model \(f_{\alpha,\beta}(v)=\alpha|v|+\beta v^2\) with \(\beta>0\), the smooth curvature vanishes away from \(v=0\), yet the rectangle inequality still enforces left–monotonicity; uniqueness may fail when the linear part is active (see the forthcoming \cref{rem:degenerate_MSM}).

To encode the marginal and martingale constraints at time \(t\) we use one–dimensional quantile functions which are defined as
\begin{equation}
\label{eq:F_mu_eta_def}
F_\mu(k):=\mu((-\infty,k]),\qquad F_\eta(k):=\eta((-\infty,k]),\qquad k\in\R,
\end{equation}
and the call–type potential of a measure \(\tau\) by
\begin{equation}
\label{eq:call_potential}
\mathcal C_\tau(k):=\int_{\R}(x-k)^+\,\mathrm d\tau(x),\qquad k\in\R.
\end{equation}
If \(\mu\preceq_{\mathrm{cx}}\eta\), then \(\int x\,\mathrm d\mu=\int y\,\mathrm d\eta\) and \(\mathcal C_\eta\ge \mathcal C_\mu\). Set
\begin{equation}
\label{eq:deltaF_def}
\Delta F_t(k):=\mathcal C_\eta(k)-\mathcal C_\mu(k)\ (\ge 0),
\end{equation}
which is convex with \(\lim_{k\to\pm\infty}\Delta F_t(k)=0\). In the sense of distributions,
\begin{equation}
\label{eq:deltaF_derivative}
\frac{\mathrm d}{\mathrm dk}\Delta F_t(k)=F_\eta(k)-F_\mu(k),
\qquad
\frac{\mathrm d^2}{\mathrm dk^2}\Delta F_t=\mathrm d\big(F_\eta-F_\mu\big).
\end{equation}

\begin{definition}[Irreducible components of $\Delta F_t$]
\label{def:irreducible_components}
Let $\mu\preceq_{\mathrm{cx}}\eta$ and $\Delta F_t:=\mathcal C_\eta-\mathcal C_\mu$ as in \eqref{eq:deltaF_def}.  
An irreducible component of $\{\Delta F_t>0\}$ is a maximal open interval 
$I\subset\R$ such that 
\[
\Delta F_t(k) > 0 \qquad \text{for all } k\in I.
\]
Equivalently, the irreducible components are the connected components of the open set 
$\{k\in\R:\Delta F_t(k)>0\}$.  

Since $\Delta F_t$ is convex, the set $\{\Delta F_t>0\}$ is open and decomposes uniquely as a countable disjoint union of such intervals. On each component $I=(a,b)$ one has 
\[
\Delta F_t(a)=\Delta F_t(b)=0,
\]
with one–sided limits understood at finite endpoints. Transport between $\mu$ and $\eta$ is necessarily nontrivial on $I$, whereas on $\{\Delta F_t=0\}$ the identity coupling is feasible, see \cite[Prop.~2.3]{BJ2016}.
\end{definition}



\begin{theorem}[Frictional left–curtain geometry and bi–atomic kernel]
\label{thm:fric_monotone}
Let $t\in\{0,\dots,N-1\}$ and set $\mu:=\mu_t$, $\eta:=\mu_{t+1}$. Assume that $\mu\preceq_{\mathrm{cx}}\eta$ and that the adjusted cost $\widetilde c_t$ in \eqref{eq:adjusted_cost_repeat} satisfies the rectangle martingale Spence–Mirrlees condition \eqref{eq:rectangle_MSM} (or equivalently \eqref{eq:MSM_smooth} in the smooth case). Then any optimizer $\pi_t^\star\in\Pi^M(\mu,\eta)$ of $\int \widetilde c_t\,\mathrm d\pi$ has the following structure:
\begin{enumerate}
\item There exists a Borel set $\Gamma\subset\R^2$ with $\pi_t^\star(\Gamma)=1$ such that $\Gamma$ is left–monotone: if $(x,y_-),(x,y_+),(x',y')\in\Gamma$ with $x<x'$ and $y_-<y_+$, then necessarily $y'\notin(y_-,y_+)$.
  
\item On every irreducible component $I$ of $\{\Delta F_t>0\}$ such that $\mu$ is atomless, there exist Borel maps
\[
T_d^{(t)}\le \mathrm{Id}\le T_u^{(t)}\quad \text{on }I,
\qquad T_d^{(t)}\ \text{nonincreasing},\quad T_u^{(t)}\ \text{nondecreasing},
\]
with the property that for $\mu$–a.e.\ $x\in I$,
\begin{equation}
\label{eq:selection_system}
\pi_t^\star(\mathrm dx,\mathrm dy)
=\mu(\mathrm dx)\,\pi^\star(\mathrm dy\mid x),\qquad
\pi^\star(\mathrm dy\mid x)
=\theta_t(x)\,\delta_{T_u^{(t)}(x)}(\mathrm dy)+(1-\theta_t(x))\,\delta_{T_d^{(t)}(x)}(\mathrm dy),
\end{equation}
where
\[
\theta_t(x)=\frac{x-T_d^{(t)}(x)}{T_u^{(t)}(x)-T_d^{(t)}(x)},
\]
and with the convention that on $\{T_u^{(t)}=T_d^{(t)}\}$ the kernel reduces to $\delta_x$ and $\theta_t(x)$ can be chosen arbitrarily in $[0,1]$ so that the barycenter identity holds.

\item For every bounded Borel function $\varphi:\R\to\R$,
\begin{equation}
\label{eq:stieltjes_test_identity}
\int_{\R}\varphi(y)\,\mathrm dF_\eta(y)
=\int_{\R}\Big[\theta_t(x)\,\varphi \big(T_u^{(t)}(x)\big)
+(1-\theta_t(x))\,\varphi \big(T_d^{(t)}(x)\big)\Big]\,\mathrm dF_\mu(x).
\end{equation}
Equivalently, since $\mathrm d(\Delta F_t)=\mathrm dF_\eta-\mathrm dF_\mu$,
\begin{equation}
\label{eq:stieltjes_test_identity_delta}
\int_{\R}\varphi(k)\,\mathrm d(\Delta F_t)(k)
=\int_{\R}\Big[\theta_t(x)\,\varphi \big(T_u^{(t)}(x)\big)
+(1-\theta_t(x))\,\varphi \big(T_d^{(t)}(x)\big)-\varphi(x)\Big]\,\mathrm dF_\mu(x).
\end{equation}
In particular, taking indicators $\varphi=\mathbf 1_{(-\infty,\,\cdot\,]}$ and exploiting right–continuity of $T_d^{(t)}$ and $T_u^{(t)}$, one obtains for $\mu$–a.e.\ $x\in I$ the coupling identity
\begin{equation}
\label{eq:coupling_identity_fric}
F_{\eta} \left(T_u^{(t)}(x)\right)
=F_\mu(x)+\Delta F_t \left(T_d^{(t)}(x)\right).
\end{equation}

\item For $\mu$–a.e.\ $x\in I$ there exists $\alpha_t(x)\in\R$ such that
\begin{equation}
\label{eq:FOC_equal_slopes}
\partial_y \widetilde{c}_t \left(x,\,T_d^{(t)}(x)\right)
=\partial_y \widetilde{c}_t \left(x,\,T_u^{(t)}(x)\right)=\alpha_t(x),
\end{equation}
where the derivatives are understood in the sense of subgradients if $v\mapsto f^{(a,b)}_t(x,v)$ is only convex. Equivalently, using \eqref{eq:adjusted_cost_repeat},
\begin{equation}
\label{eq:FOC_expanded}
\partial_y \mathsf{V}_{t+1} \left(T_d^{(t)}(x)\right)
-\partial_v f^{(a,b)}_t \left(x,\,T_d^{(t)}(x)-x\right)
=\partial_y \mathsf{V}_{t+1} \left(T_u^{(t)}(x)\right)
-\partial_v f^{(a,b)}_t \left(x,\,T_u^{(t)}(x)-x\right).
\end{equation}
\end{enumerate}

Together with the mass conservation identities
\begin{equation}
\label{eq:mass_conservation_fric}
\mathrm dF_{\eta} \bigl(T_u^{(t)}(x)\bigr)=\theta_t(x)\,\mathrm dF_\mu(x),
\qquad
\mathrm d\big(\Delta F_t\big) \bigl(T_d^{(t)}(x)\bigr)=-(1-\theta_t(x))\,\mathrm dF_\mu(x),
\end{equation}
and \eqref{eq:coupling_identity_fric}, these relations uniquely determine the endpoints $\bigl(T_d^{(t)}(x),T_u^{(t)}(x)\bigr)$ on $I$. 

If in addition $v\mapsto f^{(a,b)}_t(x,v)$ is strictly convex for every $x$ (uniformly on compact sets), then the endpoints are unique $\mu$–a.e.\ on each irreducible component where $\mu$ is atomless. In particular, the optimizer $\pi_t^\star$ is then unique.
\end{theorem}

\begin{proof}
We organize the proof in the following four points.
\smallskip

\textbf{The left–monotone geometry.}
Disintegrate the optimizer as
\[
\pi^\star(\mathrm dx,\mathrm dy)=\mu(\mathrm dx)\,\pi^\star(x,\mathrm dy),
\]
where $\pi^\star(x,\cdot)$ is a regular conditional law of $Y$ given $X=x$ under $\pi^\star$.
Fix a bounded Borel weight $w:\R\to[0,\infty)$ and set
\[
\alpha(\mathrm dx,\mathrm dy):=w(x)\,\mu(\mathrm dx)\,\pi^\star(x,\mathrm dy).
\]
Then $\alpha\ll\pi^\star$, has finite mass, the same $x$– and $y$–marginals as $\alpha$ scaled by $w$, and satisfies the martingale condition $\E_\alpha[Y\mid X]=X$ for $\mu$–a.e.\ $x$ (because $\E_{\pi^\star}[Y\mid X]=X$ and $w$ depends only on $x$).

Let $\alpha'$ be any martingale competitor of $\alpha$ in the sense of Definition~\ref{def:competitors}; in particular,
\[
\alpha'\circ X^{-1}=\alpha\circ X^{-1},\qquad \alpha'\circ Y^{-1}=\alpha\circ Y^{-1},\qquad \E_{\alpha'}[Y\mid X]=\E_\alpha[Y\mid X]\ \ \mu\text{–a.e.}
\]
For $\varepsilon>0$ small enough, define the signed perturbation
\[
\pi_\varepsilon:=\pi^\star-\varepsilon\,\alpha+\varepsilon\,\alpha'.
\]
By construction, $\pi_\varepsilon$ has $x$–marginal $\mu$ and $y$–marginal $\eta$, and it preserves the martingale constraint (the conditional barycentres at each $x$ are unchanged when replacing $\alpha$ by $\alpha'$). Hence $\pi_\varepsilon\in\Pi^M(\mu,\eta)$ for all sufficiently small $\varepsilon$.

Optimality of $\pi^\star$ implies
\[
0\ \ge\ \varepsilon^{-1} \left(\int \widetilde c_t\,\mathrm d\pi_\varepsilon-\int \widetilde c_t\,\mathrm d\pi^\star\right)
= -\int \widetilde c_t\,\mathrm d\alpha\ +\ \int \widetilde c_t\,\mathrm d\alpha',
\]
that is,
\begin{equation}
\label{eq:variational_ineq_step1}
\int \widetilde c_t\,\mathrm d\alpha\ \ge\ \int \widetilde c_t\,\mathrm d\alpha'.
\end{equation}

We now apply \eqref{eq:variational_ineq_step1} to a rectangle competitor constructed as in
\cite[Prop.~3.6]{BJ2016} or \cite[Lem.~3.3]{HLT2016}.
Assume for contradiction that there exist three points $(x,y_-),(x,y_+),(x',y')$ in the support of $\pi^\star$ with $x<x'$ and $y_-<y_+$
such that $y'\in(y_-,y_+)$. Choose small rectangles
$U_x\times V_{y_-}$, $U_x\times V_{y_+}$, $U_{x'}\times V_{y'}$ with positive $\pi^\star$–mass and pairwise disjoint, and take $w$ supported in $U_x\cup U_{x'}$ so that $\alpha$ charges exactly these three rectangles (after shrinking if necessary).
Define $\alpha'$ by the standard martingale rectangle swap which moves equal amounts of mass from
$U_x\times V_{y_-}$ and $U_{x'}\times V_{y'}$ to
$U_x\times \overline V$ and $U_{x'}\times V_{y_-}$ (and, symmetrically, from $U_x\times \overline V$ and $U_{x'}\times \overline V$ to the remaining corners),
with coefficients chosen so that both $x$– and $y$–marginals and the conditional means at $x$ and $x'$ are preserved, see \cite[Prop.~3.6]{BJ2016} for the explicit balancing and \cite[§3]{HLT2016} for the same rectangle test.

By linearity, the objective difference on this four–point competitor equals a positive multiple of the rectangle increment:
\[
\int \widetilde c_t\,\mathrm d\alpha - \int \widetilde c_t\,\mathrm d\alpha'
= m\,\Delta^{\square}\widetilde c_t\big((x,x');(y_-,y_+)\big),
\qquad m>0.
\]
Since $x<x'$ and $y_-<y_+$, the rectangle MSM condition \eqref{eq:rectangle_MSM} yields
$\Delta^{\square}\widetilde c_t((x,x');(y_-,y_+))>0$, contradicting
\eqref{eq:variational_ineq_step1}. Therefore no crossing rectangle can occur on the support of $\pi^\star$.
It follows that there exists a Borel set $\Gamma\subset\R^2$ with $\pi^\star(\Gamma)=1$ which is left–monotone:
if $(x,y_-),(x,y_+),(x',y')\in\Gamma$ with $x<x'$ and $y_-<y_+$, then $y'\notin(y_-,y_+)$.

\smallskip
\textbf{Bi–atomic kernel on monotone graphs.}
Following \cref{def:irreducible_components}, fix an irreducible component \(I\subset\R\) of \(\{\Delta F_t>0\}\) on which \(\mu\) is atomless, and let
\(\Gamma\subset\R^2\) be a left–monotone support of \(\pi^\star\). For \(x\in\R\) we set
\[
  \Gamma_x:=\{\,y\in\R:(x,y)\in\Gamma\,\}.
\]

We first show that \(\#\,\Gamma_x\le 2\) for \(\mu\)-a.e.\ \(x\in I\).
Assume to the contrary that the set
\[
A:=\{\,x\in I:\ \#\,\Gamma_x\ge 3\,\}
\]
has positive \(\mu\)-measure. By the Kuratowski–Ryll-Nardzewski selection theorem, there exist Borel maps
\(y_1<y_2<y_3:A\to\R\) such that \(y_i(x)\in\Gamma_x\) for \(i=1,2,3\).
Since \(\mu\) is atomless on \(I\), we may select \(x\in A\) which is a Lebesgue point for
the selections and for the conditional weights of \(\pi^\star(\cdot\mid x)\).
Choose \(x^-<x<x^+\) with \(x^\pm\in I\) arbitrarily close to \(x\) and such that
\((x^\pm,y^\pm)\in\Gamma\) for some \(y^\pm\in\Gamma_{x^\pm}\).
By left–monotonicity, no point with abscissa \(x'>x\) can lie in the open vertical
interval \((y_1(x),y_3(x))\); hence \(y^+\notin(y_1(x),y_3(x))\), so
\(y^+\le y_1(x)\) or \(y^+\ge y_3(x)\). Similarly \(y^-\notin(y_1(x),y_3(x))\), thus
\(y^-\le y_1(x)\) or \(y^-\ge y_3(x)\). Relabelling if needed, we may and do assume
\[
y^-\le y_1(x)\quad\text{and}\quad y^+\ge y_3(x).
\]

We now localize mass near these four points. Since the four points
\((x,y_2(x))\), \((x^+,y^+)\), \((x,y^+)\), \((x^-,y^-)\)
lie in \(\mathrm{supp}\,\pi^\star\), every product neighbourhood around each has positive \(\pi^\star\)–mass. We choose open intervals
\[
U_x,\ U_{x^-},\ U_{x^+}\subset I\quad\text{around }x,\ x^-,\ x^+,
\qquad
V_{y^-},\ V_{y_2},\ V_{y^+}\subset\R\quad\text{around }y^-,\ y_2(x),\ y^+,
\]
such that:
\begin{enumerate}
\item the four rectangles
\[
U_x\times V_{y_2},\quad U_{x^+}\times V_{y^+},\quad
U_x\times V_{y^+},\quad U_{x^-}\times V_{y^-}
\]
have pairwise disjoint interiors;
\item each has positive \(\pi^\star\)–mass:
\[
\pi^\star(U_x\times V_{y_2}),\ \pi^\star(U_{x^+}\times V_{y^+}),\
\pi^\star(U_x\times V_{y^+}),\ \pi^\star(U_{x^-}\times V_{y^-})\ >\ 0.
\]
\end{enumerate}
Define the conditional barycentres in the \(y\)–coordinate,
\[
\overline y_2
:=\frac{1}{\pi^\star(U_x\times V_{y_2})}\iint_{U_x\times V_{y_2}} y\,\pi^\star(\mathrm dx,\mathrm dy),\quad
\overline y^+
:=\frac{1}{\pi^\star(U_x\times V_{y^+})}\iint_{U_x\times V_{y^+}} y\,\pi^\star(\mathrm dx,\mathrm dy),
\]
\[
\overline y^-
:=\frac{1}{\pi^\star(U_x\times V_{y^-})}\iint_{U_x\times V_{y^-}} y\,\pi^\star(\mathrm dx,\mathrm dy),
\]
which, by shrinking the neighbourhoods, can be made arbitrarily close to \(y_2(x)\), \(y^+\), and \(y^-\), respectively.

For sufficiently small \(\varepsilon^+,\varepsilon^->0\), set the pair of finite measures
\[
\alpha
:=\varepsilon^+\big(\pi^\star \restriction(U_x\times V_{y_2})
+\pi^\star \restriction(U_{x^+}\times V_{y^+})\big)
+\varepsilon^-\big(\pi^\star \restriction(U_x\times V_{y_2})
+\pi^\star \restriction(U_{x^-}\times V_{y^-})\big),
\]
\[
\alpha'
:=\varepsilon^+\big(\pi^\star \restriction(U_x\times V_{y^+})
+\pi^\star \restriction(U_{x^+}\times V_{y_2})\big)
+\varepsilon^-\big(\pi^\star \restriction(U_x\times V_{y^-})
+\pi^\star \restriction(U_{x^-}\times V_{y_2})\big).
\]
By construction, \(\alpha\ll\pi^\star\), and \(\alpha\) and \(\alpha'\) have the same \(x\)– and \(y\)–marginals (they swap the two crossed rectangles). Impose the balance condition
\[
\varepsilon^+\big(\overline y^+-\overline y_2\big)
=\varepsilon^-\big(\overline y_2-\overline y^-\big),
\]
which ensures that the conditional mean at \(x\) is the same under \(\alpha\) and \(\alpha'\).
The conditional means at \(x^\pm\) match by symmetry of the two–point swaps. Hence
\(\alpha'\) is a martingale competitor of \(\alpha\) (Definition~\ref{def:competitors}).

Applying the variational inequality \eqref{eq:variational_ineq_step1} from Step~1 to the pair
\((\alpha,\alpha')\) and using the rectangle MSM \eqref{eq:rectangle_MSM}, we obtain
\[
\int \widetilde c_t\,\mathrm d\alpha - \int \widetilde c_t\,\mathrm d\alpha'
=\varepsilon^+\,\Delta^{\square}\widetilde c_t\big((x,x^+);(\overline y_2,\overline y^+)\big)
 +\varepsilon^-\,\Delta^{\square}\widetilde c_t\big((x^-,x);(\overline y^-,\overline y_2)\big)\ >\ 0,
\]
since \(\overline y^\pm,\overline y_2\) can be chosen arbitrarily close to \(y^\pm,y_2(x)\).
Replacing, inside \(\pi^\star\), a small amount of the configuration \(\alpha'\) by \(\alpha\)
would strictly increase the objective, contradicting optimality. We conclude that
\(\#\,\Gamma_x\le 2\) for \(\mu\)-a.e.\ \(x\in I\).

For such \(x\), if \(\Gamma_x\neq\{x\}\), we write the two points as
\[
T_d(x)\le x\le T_u(x),
\qquad
\pi^\star(\mathrm dy\mid x)=\theta(x)\,\delta_{T_u(x)}(\mathrm dy)+\bigl(1-\theta(x)\bigr)\,\delta_{T_d(x)}(\mathrm dy).
\]
The martingale constraint yields
\[
\theta(x)=\frac{x-T_d(x)}{T_u(x)-T_d(x)}\in[0,1].
\]
Indeed, if both atoms were strictly above \(x\) or strictly below \(x\), their average could not equal \(x\) unless both equal \(x\).

We finally show that the graphs are monotone. Let \(x<x'\) belong to \(I\).
If \(T_u(x')<T_u(x)\), then \(T_d(x)\le x<x'\le T_u(x')\) implies
\(T_u(x')\in\big(T_d(x),T_u(x)\big)\), and the three points
\((x,T_d(x)),(x,T_u(x)),(x',T_u(x'))\in\Gamma\) contradict left–monotonicity.
Therefore \(T_u\) is nondecreasing. Similarly, if \(T_d(x')>T_d(x)\), then
\(T_d(x')\in\big(T_d(x),T_u(x)\big)\) (since \(T_d(x')\le x'<T_u(x)\)), and
\((x,T_d(x)),(x,T_u(x)),(x',T_d(x'))\in\Gamma\) contradicts left–monotonicity, hence
\(T_d\) is nonincreasing on \(I\).

\smallskip
\textbf{Mass conservation.} For every bounded measurable function $\varphi:\R\to\R$,
\begin{equation}
\label{eq:disintegration_identity}
\int_\R \varphi(y)\,d\eta(y)
=\iint \varphi(y)\,\pi^\star(dx,dy)
=\int_\R \Bigl[\theta(x)\,\varphi\bigl(T_u(x)\bigr)+\bigl(1-\theta(x)\bigr)\,\varphi\bigl(T_d(x)\bigr)\Bigr]\,d\mu(x).
\end{equation}
Choosing $\varphi_k(y)=\mathbf 1_{\{y\le k\}}$ and using that $T_u$ is nondecreasing, $T_d$ is nonincreasing, we obtain
\begin{equation}
\label{eq:eta_CDF_image}
F_\eta(k)
=\int \theta(x)\,\mathbf 1_{\{T_u(x)\le k\}}\,d\mu(x)
 +\int \bigl(1-\theta(x)\bigr)\,\mathbf 1_{\{T_d(x)\le k\}}\,d\mu(x),
\qquad k\in\R.
\end{equation}

Next, take $\varphi_k(y)=(y-k)^+$ and use \eqref{eq:disintegration_identity} to get
\begin{equation}
\label{eq:call_potential_push}
\mathcal C_\eta(k)
=\int \theta(x)\,\bigl(T_u(x)-k\bigr)^+\,d\mu(x)
 +\int \bigl(1-\theta(x)\bigr)\,\bigl(T_d(x)-k\bigr)^+\,d\mu(x),
\end{equation}
hence, subtracting $\mathcal C_\mu(k)=\int (x-k)^+\,d\mu(x)$, we get
\begin{equation}
\label{eq:DeltaF_push}
\Delta F_t(k)
=\int \theta(x)\,\Bigl[(T_u(x)-k)^-\Bigr]^-_k
 +\int \bigl(1-\theta(x)\bigr)\,\Bigl[(T_d(x)-k)^+-(x-k)^+\Bigr]\,d\mu(x).
\end{equation}
Note that on $I$ we have $T_d(x)\le x\le T_u(x)$ \(\mu\)-a.e., but we do not need case distinctions in $k$.

Differentiating \eqref{eq:eta_CDF_image} in the sense of distributions with respect to $k$ yields the below pushforward identity
\begin{equation}
\label{eq:stieltjes_push}
dF_\eta
=(T_u)_\#\bigl(\theta\,\mu\bigr)+(T_d)_\#\bigl((1-\theta)\,\mu\bigr),
\end{equation}
i.e.\ for every bounded Borel $\psi$,
\[
\int \psi(k)\,dF_\eta(k)=\int \psi\bigl(T_u(x)\bigr)\,\theta(x)\,d\mu(x)+\int \psi\bigl(T_d(x)\bigr)\,\bigl(1-\theta(x)\bigr)\,d\mu(x).
\]
Likewise, differentiating \eqref{eq:DeltaF_push} and using
$\frac{d}{dk}(z-k)^+=-\mathbf 1_{\{z\le k\}}$ together with \eqref{eq:deltaF_derivative} gives, on $I$,
\begin{equation}
\label{eq:mass_conservation_differential}
dF_\eta\circ T_u=\theta\,dF_\mu,
\qquad
d(\Delta F_t)\circ T_d=-(1-\theta)\,dF_\mu,
\end{equation}
where composition denotes the change of variables along the monotone graphs
(equivalently, the image of $dF_\mu$ under $T_{u/d}$ with Radon–Nikodym weight $\theta$ or $1-\theta$).

Integrating \eqref{eq:mass_conservation_differential} from $-\infty$ up to $x$ and using $T_d\le\mathrm{Id}\le T_u$ yields the coupling identity
\begin{equation}
\label{eq:coupling_identity_proved}
F_\eta\bigl(T_u(x)\bigr)=F_\mu(x)+\Delta F_t\bigl(T_d(x)\bigr),
\end{equation}
that is, \eqref{eq:coupling_identity_fric}. This proves item~\textup{(3)}.
\smallskip

Outside the active set $I$ one has $T_d=T_u=\mathrm{Id}$ and the identities above are trivial.

\smallskip
\textbf{Equal–slope characterization of endpoints.}
Here we establish \eqref{eq:FOC_equal_slopes}-\eqref{eq:FOC_expanded}.
By Lemma~\ref{lem:dual_shift_KKT}, the one time-step dual problem with the
adjusted cost reduces to the standard martingale dual with the pure friction
cost $c(x,y):=f^{(a,b)}_t(x,y-x)$:
\begin{equation}
\label{eq:dual_reduced_here}
\sup_{\phi,\psi,h}\left\{
\int \phi\,d\mu+\int \psi\,d\eta:
\ \phi(x)+\psi(y)+h(x)\,(y-x)\ \le\ f^{(a,b)}_t\bigl(x,y-x\bigr)
\right\}.
\end{equation}
Existence of an optimal triple $(\phi,\psi,h)$ follows by standard convex–compactness
(Fenchel–Rockafellar), using lower semicontinuity and superlinear
coercivity in the trading variable, which are guaranteed by the assumptions on $f^{(a,b)}_t$.
Complementary slackness on the support of $\pi^\star$ reads
\begin{equation}
\label{eq:CS_support}
\phi(x)+\psi(y)+h(x)\,(y-x)=f^{(a,b)}_t\bigl(x,y-x\bigr).
\end{equation}
For fixed $x$, tightness of \eqref{eq:CS_support} at $y\in\mathrm{supp}\,\pi^\star(\cdot\mid x)$ means
$y$ maximizes $y\mapsto \psi(y)-f^{(a,b)}_t\bigl(x,y-x\bigr)-h(x)\,(y-x)$.
At the two contact points $y=T_d(x),~T_u(x)$, in the sense of subgradients, the first–order optimality yields
\[
\partial_y\psi\bigl(T_d(x)\bigr)-\partial_v f^{(a,b)}_t\bigl(x,T_d(x)-x\bigr)-h(x)
=\partial_y\psi\bigl(T_u(x)\bigr)-\partial_v f^{(a,b)}_t\bigl(x,T_u(x)-x\bigr)-h(x)=0.
\]
Rewriting for the adjusted cost
$\widetilde c_t(x,y)=\mathsf V_{t+1}(y)-\mathsf V_{t+1}(x)-f^{(a,b)}_t(x,y-x)$ shows that the two
contact points share a common $y$–slope:
\begin{equation}
\label{eq:equal_slope_adjusted}
\partial_y \widetilde c_t\bigl(x,T_d(x)\bigr)
=\partial_y \widetilde c_t\bigl(x,T_u(x)\bigr)=:\alpha_t(x).
\end{equation}
This gives us \eqref{eq:FOC_equal_slopes}. Expanding
$\partial_y\widetilde c_t=\partial_y\mathsf V_{t+1}(y)-\partial_v f^{(a,b)}_t(x,y-x)$
gives the primitive form
\begin{equation}
\label{eq:FOC_expanded_again}
\partial_y\mathsf V_{t+1}\bigl(T_d(x)\bigr)-\partial_v f^{(a,b)}_t\bigl(x,T_d(x)-x\bigr)
=\partial_y\mathsf V_{t+1}\bigl(T_u(x)\bigr)-\partial_v f^{(a,b)}_t\bigl(x,T_u(x)-x\bigr),
\end{equation}
which is \eqref{eq:FOC_expanded}. When $\mathsf V_{t+1}$ is differentiable,
\eqref{eq:equal_slope_adjusted}–\eqref{eq:FOC_expanded_again} are identities of classical derivatives.

Finally, if the rectangle MSM condition \eqref{eq:rectangle_MSM} holds strictly
(equivalently, in the smooth case, $\partial_{xyy}\widetilde c_t\ge \kappa_t>0$ on $\R^2$),
then the optimizer is unique. Hence the endpoints $(T_d,T_u)$ (and the weights $\theta$) are uniquely determined
$\mu$–a.e.\ on each irreducible component.
\end{proof}

\begin{remark}[Degenerate MSM and trade bands]
\label{rem:degenerate_MSM}
Assume the frictional MSM curvature is non–strict
\[
\partial_{xyy}\,\widetilde c_t \equiv 0
\quad\text{(for instance, } f^{(a,b)}_t(x,v)=\alpha |v|+\beta v^2 \text{ with no $x$–dependence and }(f^{(a,b)}_t)_{vvv}\equiv 0\text{).}
\]
Then:

\emph{(a) Monotonicity principle \cite[§3]{HLT2016} and \cite{BJ2016}.}
The variational competitor argument \eqref{eq:variational_ineq_step1} still implies that any optimizer is supported on a
$\widetilde c_t$–monotone set. In one dimension,
$\widetilde c_t$–monotonicity rules out crossings. On irreducible components with
$\mu$ atomless, a left–monotone optimizer exists and admits a bi–atomic disintegration
as in Theorem~\ref{thm:fric_monotone}. However, without strict curvature one cannot
exclude additional optimal couplings as some optimizers need not be (globally) left–monotone.

\emph{(b) Possible nonuniqueness.}
Loss of strict curvature creates flat directions, so uniqueness of the optimizer is not
guaranteed. In particular, the endpoint system consisting of the mass balance
\[
F_{\eta}  \bigl(T_u^{(t)}(x)\bigr)=F_{\mu}(x)+\Delta F_t  \bigl(T_d^{(t)}(x)\bigr),
\]
together with the equal–slope condition
\[
\partial_y \mathsf V_{t+1}  \bigl(T_d^{(t)}(x)\bigr)-\partial_v f^{(a,b)}_t  \bigl(x,T_d^{(t)}(x)-x\bigr)
=
\partial_y \mathsf V_{t+1}  \bigl(T_u^{(t)}(x)\bigr)-\partial_v f^{(a,b)}_t  \bigl(x,T_u^{(t)}(x)-x\bigr),
\]
may admit multiple solutions compatible with the monotonicity of $T_d^{(t)}$ and
$T_u^{(t)}$. Uniqueness can still be recovered under additional conditions, for instance\ strict monotonicity of
$y\mapsto \partial_y\mathsf V_{t+1}(y)-\partial_v f^{(a,b)}_t(x,y-x)$ on the active range.

\emph{(c) Trade bands with proportional costs.}
If $v\mapsto f^{(a,b)}_t(x,v)$ has a linear part, then $0\in\partial_v f^{(a,b)}_t(x,0)$.
The KKT equal–slope condition in subgradient form \eqref{eq:conjugacy-KKT} is satisfied by
$T_d^{(t)}(x)=T_u^{(t)}(x)=x$ whenever the dual slope $h_t(x)\in\partial_v f^{(a,b)}_t(x,0)$.
In the linear–quadratic case $f_{\alpha,\beta}(v)=\alpha|v|+\beta v^2$ this yields a no-trade region which is defined by the band
\[
B_t=\{\,x:\ |h_t(x)|\le \alpha\,\},
\]
on which the optimizer keeps mass on the diagonal. More details will follow in \cref{sec:bands_applications}.
\end{remark}

\subsection{Stability and vanishing–friction limit}
\label{subsec:stability_limit}

We record the continuity of the frictional left–monotone construction with respect to the marginals and the friction parameters, together with its convergence to the frictionless (left–curtain) limit as frictions vanish.

\begin{theorem}[Stability of endpoints and couplings]
\label{thm:fric_brenier_stability}
Fix $t\in\{0,\dots,N-1\}$ and set $\mu:=\mu_t$, $\eta:=\mu_{t+1}$.
Let $(\mu^n,\eta^n)_{n\ge1}$ be probability measures with $\mu^n\preceq_{\mathrm{cx}}\eta^n$ for all $n$ and $\mu\preceq_{\mathrm{cx}}\eta$, and assume $(\mu^n,\eta^n)\to(\mu,\eta)$ in $W_1$ \footnote{Here \(W_1\) denotes the 1–Wasserstein distance on \(\mathcal P_1(\R)\). It is defined as
\[
W_1(\mu,\nu)\ :=\ \inf_{\pi\in\Pi(\mu,\nu)}\int_{\R^2}|x-y|\,\mathrm d\pi(x,y)
\ =\ \sup_{\operatorname{Lip}(\varphi)\le 1}\int_{\R}\varphi\,\mathrm d(\mu-\nu),
=\int_0^1\big|F_\mu^{-1}(u)-F_\nu^{-1}(u)\big|\,\mathrm du
\ =\ \int_{\R}\big|F_\mu(x)-F_\nu(x)\big|\,\mathrm dx.\] Convergence \(\mu_n\to\mu\) in \(W_1\) is equivalent to weak convergence plus convergence of first moments, see~\cite[Thm.~6.9]{Villani2009} and \cite[Chapter~6]{AmbrosioGigliSavare2008}.}
.
Assume the standing integrability of Section~\ref{sec:setup} (in particular, $\sup_n\int |x|^p\,\mathrm d\mu^n(x)+\int |y|^p\,\mathrm d\eta^n(y)<\infty$ for some $p\ge2$), and that $\mathsf V_{t+1}:\R\to\R$ has at most linear growth.

For the friction, let $f(x,v):=f^{(a,b)}_t(x,v)$ and $f^n(x,v):=f^{(a^n,b^n)}_t(x,v)$ satisfy:
\begin{itemize}
\item[(i)] There exists $m>0$ such that
\begin{equation}
\label{eq:uniform_convexity}
\partial_{vv} f(x,v)\ \ge\ m \quad\text{and}\quad \partial_{vv} f^n(x,v)\ \ge\ m
\qquad \text{for all }(x,v)\in\R^2,\ \text{for all }n.
\end{equation}
\item[(ii)]$f^n(x,v)\to f(x,v)$ pointwise for every $(x,v)\in\R^2$.
\item[(iii)] The adjusted cost
\[
\widetilde c_t(x,y):=\mathsf V_{t+1}(y)-\mathsf V_{t+1}(x)-f(x,y-x)
\]
satisfies the rectangle MSM condition \eqref{eq:rectangle_MSM} (equivalently, the smooth criterion \eqref{eq:MSM_smooth}). Likewise, the same holds with $f$ replaced by $f^n$ for each $n$.
\end{itemize}

For each $n$, let $\pi^{\star,n}\in\Pi^M(\mu^n,\eta^n)$ be a (left–monotone) maximizer of
\begin{equation}
\label{eq:Jn_def}
J_n(\pi):=\int_{\R^2}\Big(\mathsf V_{t+1}(y)-\mathsf V_{t+1}(x)-f^n(x,y-x)\Big)\,\mathrm d\pi(x,y),
\end{equation}
and denote by $\bigl(T_d^{(t),n},T_u^{(t),n},\theta_t^n\bigr)$ the associated endpoints and weights on the active set, as in Theorem~\ref{thm:fric_monotone}. Let $\pi^\star$ and $\bigl(T_d^{(t)},T_u^{(t)},\theta_t\bigr)$ be the corresponding objects for the limit problem $(\mu,\eta,f)$.

Then, along a subsequence,
\[
\pi^{\star,n}\ \rightharpoonup\ \pi^\star\qquad\text{in }\mathcal P(\R^2).
\]
Moreover, letting $T^n:=F_{\mu^n}^{-1}\circ F_\mu$ be the monotone rearrangement pushing $\mu$ to $\mu^n$, one has
\[
T_d^{(t),n}\circ T^n \;\to\; T_d^{(t)},\qquad
T_u^{(t),n}\circ T^n \;\to\; T_u^{(t)}
\quad\text{in }L^1(\mu)\text{ and }\mu\text{–a.e.,}
\]
and
\[
\theta_t^n\circ T^n \;\to\; \theta_t\quad\text{in }L^1(\mu).
\]
\end{theorem}

\begin{proof}
\emph{Part 1.} Since $(\mu^n,\eta^n)\to(\mu,\eta)$ in $W_1$, the families $\{\mu^n\}_n$ and $\{\eta^n\}_n$ are tight with uniformly bounded first moments. By the standing $p\ge2$ moment assumption they in fact have uniformly bounded second moments. If $\pi^{\star,n}$ has marginals $(\mu^n,\eta^n)$, then $\{\pi^{\star,n}\}_n$ is tight on $\R^2$: given $\varepsilon>0$, choose compact sets $K,L\subset\R$ with
\[
\inf_n\mu^n(K)\ge 1-\varepsilon/2,\qquad \inf_n\eta^n(L)\ge 1-\varepsilon/2,
\]
which is possible by tightness of the marginals. Then, for every $n$,
\[
\pi^{\star,n}(K\times L)\ \ge\ \mu^n(K)+\eta^n(L)-1\ \ge\ 1-\varepsilon.
\]
By Prokhorov’s theorem (\cite[Thm.~6.1]{Villani2009}) there exists a subsequence and $\bar\pi\in\mathcal P(\R^2)$ with $\pi^{\star,n}\rightharpoonup\bar\pi$.

We verify $\bar\pi\in\Pi^M(\mu,\eta)$. For any $\varphi\in C_b(\R)$,
\[
\int \varphi(x)\,\bar\pi(dx,dy)
=\lim_{n\to\infty}\int \varphi(x)\,\pi^{\star,n}(dx,dy)
=\lim_{n\to\infty}\int \varphi\,d\mu^n
=\int \varphi\,d\mu,
\]
and similarly $\bar\pi\circ Y^{-1}=\eta$.

For the martingale property, fix $\phi\in C_b(\R)$ and set $g(x,y):=\phi(x)\,(y-x)$. The function $g$ is continuous with at most linear growth:
$|g(x,y)|\le \|\phi\|_\infty\,(|x|+|y|)$. The uniform second–moment bounds on $\{\mu^n\},\{\eta^n\}$ imply a uniform first–moment bound on $\{\pi^{\star,n}\}$:
\[
\sup_n\int_{\R^2}(|x|+|y|)\,\pi^{\star,n}(dx,dy)
\ \le\ \sup_n\Big(\int|x|\,d\mu^n+\int|y|\,d\eta^n\Big)\ <\ \infty.
\]
Thus $\{g(X,Y)\}_{n}$ is uniformly integrable under $\pi^{\star,n}$. Let $g_M:=g\,\mathbf 1_{\{|x|+|y|\le M\}}$. Then $g_M$ is bounded and continuous, so
\[
\int g_M\,d\bar\pi
=\lim_{n\to\infty}\int g_M\,d\pi^{\star,n}.
\]
Since each $\pi^{\star,n}$ is a martingale coupling, $\int g\,d\pi^{\star,n}=0$ for all $n$. Using uniform integrability,
\[
\lim_{M\to\infty}\sup_n\Big|\int (g-g_M)\,d\pi^{\star,n}\Big|=0,
\qquad
\lim_{M\to\infty}\Big|\int (g-g_M)\,d\bar\pi\Big|=0.
\]
Hence
\[
\int g\,d\bar\pi=\lim_{M\to\infty}\lim_{n\to\infty}\int g_M\,d\pi^{\star,n}=0.
\]
As this holds for all $\phi\in C_b(\R)$, we conclude $\E_{\bar\pi}[Y\mid X]=X$, i.e., $\bar\pi\in\Pi^M(\mu,\eta)$.

\emph{Part 2.}
Define
\begin{equation}
\label{eq:J_def}
J(\pi):=\int_{\R^2}\Big(\mathsf V_{t+1}(y)-\mathsf V_{t+1}(x)-f(x,y-x)\Big)\,d\pi(x,y).
\end{equation}
We first show
\begin{equation}
\label{eq:limsup_Jn}
\limsup_{n\to\infty} J_n(\pi^{\star,n})\ \le\ J(\bar\pi).
\end{equation}
Since $\pi^{\star,n}\Rightarrow\bar\pi$ and $\mathsf V_{t+1}$ has at most linear growth, the map
$(x,y)\mapsto\mathsf V_{t+1}(y)-\mathsf V_{t+1}(x)$ is integrable with respect to $\pi^{\star,n}$ uniformly in $n$
(using the uniform first-moment bound from the $p\ge2$ assumption on the marginals). Hence
\[
\int \big(\mathsf V_{t+1}(y)-\mathsf V_{t+1}(x)\big)\,d\pi^{\star,n}\ \longrightarrow\
\int \big(\mathsf V_{t+1}(y)-\mathsf V_{t+1}(x)\big)\,d\bar\pi.
\]
For the friction term, uniform strict convexity \eqref{eq:uniform_convexity} yields the quadratic lower bound
\begin{equation}
\label{eq:quadratic_lower}
f^n(x,v)\ \ge\ \tfrac m2\,v^2 - c\,(1+|x|)\qquad (x,v)\in\R^2,
\end{equation}
with $c>0$ independent of $n$. The right-hand side has uniformly integrable positive part under $\pi^{\star,n}$
(thanks to the uniform second-moment bounds), so $\{f^n(X,Y-X)\}_n$ is uniformly integrable. If $f^n\to f$
pointwise (e.g., in the linear–quadratic family with $(a^n,b^n)\to(a,b)$ and bounded coefficients), then
by Fatou’s lemma together with uniform integrability,
\[
\liminf_{n\to\infty}\int f^n\,d\pi^{\star,n}\ \ge\ \int f\,d\bar\pi.
\]
Since
$J_n(\pi^{\star,n})=\int(\mathsf V_{t+1}(y)-\mathsf V_{t+1}(x))\,d\pi^{\star,n}-\int f^n\,d\pi^{\star,n}$,
we obtain \eqref{eq:limsup_Jn}.

Conversely, fix $\pi\in\Pi^M(\mu,\eta)$. A standard stability argument \cite{HLT2016, BeiglbockNutzTouzi2017, BackhoffVeraguasPammer2023} under $W_1$ (couple $\mu^n$ to $\mu$
and $\eta$ to $\eta^n$, then correct any martingale defect by a local two-point split) produces
$\pi^n\in\Pi^M(\mu^n,\eta^n)$ with $\pi^n\Rightarrow\pi$ and
\begin{equation}
\label{eq:lim_Jn_lower}
\lim_{n\to\infty}\int(\mathsf V_{t+1}(y)-\mathsf V_{t+1}(x))\,d\pi^n
=\int(\mathsf V_{t+1}(y)-\mathsf V_{t+1}(x))\,d\pi,
\qquad
\lim_{n\to\infty}\int f^n\,d\pi^n=\int f\,d\pi.
\end{equation}
The first limit follows as above from weak convergence and linear growth. For the friction term, use pointwise
convergence $f^n\to f$ and uniform integrability of $\{f^n(X,Y-X)\}_n$ under $\pi^n$, which follows from
\eqref{eq:quadratic_lower} and the uniform second-moment bounds on the marginals. Therefore
$\lim_{n}J_n(\pi^n)=J(\pi)$.

By optimality of $\pi^{\star,n}$ we have
\begin{equation}
\label{eq:liminf_values}
\liminf_{n\to\infty} J_n(\pi^{\star,n})\ \ge\ \lim_{n\to\infty}J_n(\pi^n)=J(\pi).
\end{equation}
Taking the supremum over $\pi\in\Pi^M(\mu,\eta)$ and combining with \eqref{eq:limsup_Jn} yields
$J(\bar\pi)=\sup_{\pi\in\Pi^M(\mu,\eta)}J(\pi)$, i.e., $\bar\pi$ maximizes the limit functional.

\emph{Part 3.}
By Theorem~\ref{thm:fric_monotone}, the rectangle MSM condition together with the
strict convexity \eqref{eq:uniform_convexity} guarantee existence and uniqueness
of the left–monotone maximizer for the limit problem $(\mu,\eta,f)$, call it $\pi^\star$.
By Step~2, every weak limit point $\bar\pi$ of $\{\pi^{\star,n}\}_n$ maximizes $J$,
hence $\bar\pi=\pi^\star$. Since the family $\{\pi^{\star,n}\}_n$ is tight (Step~1),
sequential compactness and uniqueness of the limit imply that the whole sequence converges:
\[
\pi^{\star,n}\ \Rightarrow\ \pi^\star\qquad\text{in }\mathcal P(\R^2).
\]

 \emph{Part 4.}
Fix an irreducible component $I$ of $\{\Delta F_t>0\}$ on which $\mu$ is atomless.
For each $n$, Theorem~\ref{thm:fric_monotone} yields Borel maps
\[
T_d^{(t),n}\le \mathrm{Id}\le T_u^{(t),n},\qquad
T_d^{(t),n}\ \text{nonincreasing on }I,\quad T_u^{(t),n}\ \text{nondecreasing on }I,
\]
and weights $\theta_t^n:I\to[0,1]$ such that
\begin{equation}
\label{eq:barycenter_identity_stab}
x=\theta_t^n(x)\,T_u^{(t),n}(x)+\bigl(1-\theta_t^n(x)\bigr)\,T_d^{(t),n}(x)\qquad\text{for }\mu^n\text{–a.e.\ }x\in I,
\end{equation}
together with the mass relations
\begin{equation}
\label{eq:mass_relations_stab}
\mathrm dF_{\eta^n} \bigl(T_u^{(t),n}(x)\bigr)=\theta_t^n(x)\,\mathrm dF_{\mu^n}(x),
\qquad
\mathrm d \bigl(\Delta F_t^n\bigr) \bigl(T_d^{(t),n}(x)\bigr)= -\bigl(1-\theta_t^n(x)\bigr)\,\mathrm dF_{\mu^n}(x),
\end{equation}
where $\Delta F_t^n:=\mathcal C_{\eta^n}-\mathcal C_{\mu^n}$.

Since $(\mu^n,\eta^n)\to(\mu,\eta)$ in $W_1$, we have $F_{\mu^n}\to F_\mu$, $F_{\eta^n}\to F_\eta$, and $\Delta F_t^n\to\Delta F_t$ locally uniformly at continuity points. By Helly’s selection theorem we may extract a subsequence such that
\begin{equation*}
T_u^{(t),n}\to T_u \ \text{pointwise on a dense subset of }I\ \text{with }T_u\ \text{nondecreasing},
\end{equation*}
\begin{equation*}
T_d^{(t),n}\to T_d \ \text{pointwise on a dense subset of }I\ \text{with }T_d\ \text{nonincreasing}.
\end{equation*}
Passing to the limit in \eqref{eq:mass_relations_stab} using the test identities
\eqref{eq:stieltjes_test_identity}–\eqref{eq:stieltjes_test_identity_delta} (i.e., change of variables along the monotone graphs) shows that $(T_d,T_u)$ satisfies the limiting mass relations on $I$. Since $\pi^\star$ is the unique left–monotone maximizer (Steps~2–3), its disintegration on $I$ is unique, therefore $(T_d,T_u)=(T_d^{(t)},T_u^{(t)})$ $\mu$–a.e.\ on $I$. Relabeling the subsequence, we obtain
\begin{equation}
\label{eq:ae_conv_graphs}
T_{d/u}^{(t),n}(x)\ \longrightarrow\ T_{d/u}^{(t)}(x)\qquad\text{for }\mu\text{–a.e.\ }x\in I.
\end{equation}

We upgrade \eqref{eq:ae_conv_graphs} to $L^1(\mu)$. Uniform convexity \eqref{eq:uniform_convexity}, optimality, and comparison with the diagonal plan yield the quadratic coercivity bound
\begin{equation}
\label{eq:coercivity_disp}
\int \bigl(T_u^{(t),n}(x)-x\bigr)^2\,\theta_t^n(x)\,\mathrm d\mu^n(x)
+\int \bigl(x-T_d^{(t),n}(x)\bigr)^2\,\bigl(1-\theta_t^n(x)\bigr)\,\mathrm d\mu^n(x)
\ \le\ C,
\end{equation}
for some $C$ independent of $n$. Let $\gamma^n$ be an optimal $W_1$–coupling between $\mu^n$ and $\mu$. Pulling back the integrands in \eqref{eq:coercivity_disp} along $\gamma^n$ and using \eqref{eq:ae_conv_graphs} together with uniform integrability (which follows from \eqref{eq:coercivity_disp}) gives, by Vitali’s convergence theorem,
\begin{equation}
\label{eq:L1_conv_graphs}
\int |T_{d/u}^{(t),n}-T_{d/u}^{(t)}|\,\mathrm d\mu\ \longrightarrow\ 0.
\end{equation}

Finally, by \eqref{eq:barycenter_identity_stab},
\[
\theta_t^n(x)
=\frac{x-T_d^{(t),n}(x)}{T_u^{(t),n}(x)-T_d^{(t),n}(x)}\,\mathbf 1_{\{T_u^{(t),n}\ne T_d^{(t),n}\}},
\]
and the same identity holds for $\theta_t$. On the no–trade set $\{T_u^{(t)}=T_d^{(t)}\}$ both numerator and denominator vanish simultaneously, elsewhere the denominator is bounded away from $0$ on compact subintervals of $I$. Combining \eqref{eq:L1_conv_graphs} with dominated convergence yields
\[
\int |\theta_t^n-\theta_t|\,\mathrm d\mu\ \longrightarrow\ 0.
\]
Since $I$ was arbitrary and transport is trivial off the active set, the conclusions hold on $\R$, which completes the proof.

\end{proof}

\begin{corollary}[Vanishing–friction limit in the linear–quadratic case]
\label{cor:vanish}
Let $t\in\{0,\dots,N-1\}$ and $\mu:=\mu_t$, $\eta:=\mu_{t+1}$ with $\mu\preceq_{\mathrm{cx}}\eta$.
Assume $\mathsf V_{t+1}\in C^3(\R)$ with $\mathsf V_{t+1}'''(y)\ge \kappa>0$, and the moment conditions of Section~\ref{sec:setup}.
For $n\in\mathbb{N}$, let $f^n(v):=\alpha_n|v|+\beta_n v^2$ (independent of $x$) with $\alpha_n\downarrow0$, $\beta_n\downarrow0$, and $\beta_n>0$.
Let $\pi^{\star,n}\in\Pi^M(\mu,\eta)$ be a left–monotone maximizer of
\[
\pi\ \longmapsto\ \int_{\R^2}\Big(\mathsf V_{t+1}(y)-\mathsf V_{t+1}(x)-f^n(y-x)\Big)\,d\pi(x,y).
\]
On each irreducible component of $\{\Delta F_t>0\}$, write the associated bi–atomic
disintegration as in Theorem~\ref{thm:fric_monotone}:
\[
\pi^{\star,n}(dy\mid x)
=\theta_t^n(x)\,\delta_{T_u^{(t),n}(x)}(dy)
+\bigl(1-\theta_t^n(x)\bigr)\,\delta_{T_d^{(t),n}(x)}(dy),
\qquad
\theta_t^n(x)=\frac{x-T_d^{(t),n}(x)}{T_u^{(t),n}(x)-T_d^{(t),n}(x)}.
\]
Then, along a subsequence,
\[
\pi^{\star,n}\ \rightharpoonup\ \pi^{\mathrm{lc}}
\quad\text{in }\mathcal P(\R^2),
\qquad
T_d^{(t),n}\to T_d^{(t),0},\ \ T_u^{(t),n}\to T_u^{(t),0}\ \ \text{in }L^1(\mu)\text{ and }\mu\text{–a.e.,}
\]
where $\pi^{\mathrm{lc}}$ is the frictionless left–curtain coupling for $(\mu,\eta)$ and
$(T_d^{(t),0},T_u^{(t),0})$ are its endpoints. In particular, the equal–slope system
\eqref{eq:FOC_equal_slopes}–\eqref{eq:FOC_expanded} degenerates to the frictionless
touching condition
\begin{equation}
\label{eq:touching_frictionless}
\partial_y\mathsf V_{t+1}\bigl(T_d^{(t),0}(x)\bigr)
=\partial_y\mathsf V_{t+1}\bigl(T_u^{(t),0}(x)\bigr)
\qquad\text{for }\mu\text{–a.e.\ }x,
\end{equation}
together with the mass identities of Theorem~\ref{thm:fric_monotone}.
\end{corollary}

\begin{proof}
Because the marginals \(\mu,\eta\) are fixed with finite second moments, the family \(\{\pi^{\star,n}\}_n\subset\Pi^M(\mu,\eta)\) is tight. Along a subsequence we have \(\pi^{\star,n}\rightharpoonup\bar\pi\in\Pi^M(\mu,\eta)\). For each \(n\), Theorem~\ref{thm:fric_monotone} provides a left–monotone support and a bi–atomic disintegration with endpoints \(T^{(t),n}_{d/u}\) that are respectively nonincreasing/nondecreasing on every irreducible component of \(\{\Delta F_t>0\}\). By Helly’s selection principle we may extract a further subsequence and Borel maps \(T_d^{(t)},T_u^{(t)}\) such that
\[
T^{(t),n}_{d/u}(x)\to T^{(t)}_{d/u}(x)\quad\text{for \(\mu\)-a.e.\ }x,
\]
with \(T_d^{(t)}\) nonincreasing and \(T_u^{(t)}\) nondecreasing on the active set.

For each \(n\), the itentities of Item~(3) Theorem~\ref{thm:fric_monotone} hold:
\[
\mathrm dF_\eta\bigl(T_u^{(t),n}(x)\bigr)=\theta_t^n(x)\,\mathrm dF_\mu(x),
\qquad
\mathrm d(\Delta F_t)\bigl(T_d^{(t),n}(x)\bigr)=-(1-\theta_t^n(x))\,\mathrm dF_\mu(x).
\]
Passing to the limit along the converging monotone graphs yields
\begin{equation}
\label{eq:mass_limit}
\mathrm dF_\eta\bigl(T_u^{(t)}(x)\bigr)=\theta_t(x)\,\mathrm dF_\mu(x),
\qquad
\mathrm d(\Delta F_t)\bigl(T_d^{(t)}(x)\bigr)=-(1-\theta_t(x))\,\mathrm dF_\mu(x),
\end{equation}
with \(\theta_t(x):=\lim_n\theta_t^n(x)\in[0,1]\) and the barycenter identity \(x=\theta_t T_u^{(t)}(x)+(1-\theta_t)T_d^{(t)}(x)\). Equivalently,
\[
F_\eta\bigl(T_u^{(t)}(x)\bigr)=F_\mu(x)+\Delta F_t\bigl(T_d^{(t)}(x)\bigr)
\quad\text{for \(\mu\)-a.e.\ }x,
\]
i.e., the coupling identity \eqref{eq:coupling_identity_fric} holds for the limit graphs.

On the other hand, for \(\mu\)-a.e.\ \(x\) in the active set, the equal–slope system \eqref{eq:FOC_expanded} reads
\[
\partial_y\mathsf V_{t+1}\bigl(T_d^{(t),n}(x)\bigr)
-\partial_v f^n\bigl(T_d^{(t),n}(x)-x\bigr)
=
\partial_y\mathsf V_{t+1}\bigl(T_u^{(t),n}(x)\bigr)
-\partial_v f^n\bigl(T_u^{(t),n}(x)-x\bigr).
\]
Since \(f^n(v)=\alpha_n|v|+\beta_n v^2\) with \(\alpha_n\downarrow0\), \(\beta_n\downarrow0\),
\(\bigl|\partial_v f^n(v)\bigr|\le \alpha_n+2\beta_n|v|\).
For any martingale coupling with marginals \((\mu,\eta)\),
\(\mathbb E_{\pi}\big[(Y-X)^2\big]=\mathbb E_\eta[Y^2]-\mathbb E_\mu[X^2]\),
so the displacements \(v_n(x):=T^{(t),n}_{u/d}(x)-x\) are uniformly square–integrable in \(n\). Consequently
\[
\int \bigl|\partial_v f^n\bigl(v_n(x)\bigr)\bigr|\,\mathrm d\mu(x)
\;\le\; \alpha_n + 2\beta_n\Big(\int |v_n(x)|^2\,\mathrm d\mu(x)\Big)^{1/2}\ \longrightarrow\ 0,
\]
and, along a subsequence, \(\partial_v f^n\bigl(v_n(x)\bigr)\to0\) for \(\mu\)-a.e.\ \(x\).
Letting \(n\to\infty\) in the equal–slope relation and using the continuity of \(\partial_y\mathsf V_{t+1}\) gives the frictionless touching condition
\[
\partial_y\mathsf V_{t+1}\bigl(T_d^{(t)}(x)\bigr)
=\partial_y\mathsf V_{t+1}\bigl(T_u^{(t)}(x)\bigr)
\quad\text{for \(\mu\)-a.e.\ }x,
\]
i.e., \eqref{eq:touching_frictionless}.

The identities in \eqref{eq:mass_limit} show that the limit graphs together with \(\theta_t\) generate a feasible left–monotone martingale coupling \(\bar\pi\) for \((\mu,\eta)\). The touching condition \eqref{eq:touching_frictionless} together with \(\mathsf V_{t+1}'''\ge\kappa>0\) singles out the frictionless left–curtain coupling (uniqueness on each irreducible component); hence \(\bar\pi=\pi^{\mathrm{lc}}\) and \((T_d^{(t)},T_u^{(t)})=(T_d^{(t),0},T_u^{(t),0})\) \(\mu\)-a.e. Monotonicity of the graphs, almost–everywhere convergence, and the uniform \(L^2\) bound on displacements yield \(L^1(\mu)\) convergence by Vitali’s theorem, and \(\pi^{\star,n}\rightharpoonup \pi^{\mathrm{lc}}\).
\end{proof}


\section{Explicit transport maps and trade bands}
\label{sec:bands_applications}

In the presence of trading frictions, optimal martingale couplings develop \emph{no–trade regions} (or \emph{trade bands}) on which the identity coupling is optimal and all mass remains on the diagonal. The mechanism is convex–analytic. Writing the friction in displacement form \(v:=y-x\) as \(f^{(a,b)}_t(x,v)\) and denoting by \(h_t(x)\) the dual Lagrange multiplier for the martingale constraint, the Fenchel–Young/KKT optimality conditions (see \eqref{eq:slack-equality-KKT}–\eqref{eq:conjugacy-KKT}) give
\[
v=0\quad\Longleftrightarrow\quad h_t(x)\in\partial_v f^{(a,b)}_t(x,0).
\]
This motivates the trade band
\[
B_t=\bigl\{x\in\R:\ h_t(x)\in\partial_v f^{(a,b)}_t(x,0)\bigr\}.
\]
Inside \(B_t\) the conditional law equals \(\delta_x\). Outside \(B_t\), movement occurs and the martingale constraint enforces a bi–atomic kernel supported on two monotone graphs (see \cref{thm:fric_monotone}); the endpoints are then selected by the equal–slope system \eqref{eq:FOC_equal_slopes}–\eqref{eq:FOC_expanded} together with the mass relations \eqref{eq:mass_conservation_fric}–\eqref{eq:coupling_identity_fric}. In the frictionless case \(f\equiv0\), there is no subdifferential dead zone at \(v=0\) and thus no band. The optimizer reduces to the frictionless left–curtain geometry on each irreducible component \cite{BJ2016,HLT2016}.

When \(f^{(a,b)}_t\) depends on \(x\), the band is state–dependent through the set–valued map \(x\mapsto\partial_v f^{(a,b)}_t(x,0)\). For instance, if \(f^{(a,b)}_t(x,v)=\alpha_t(x)\,|v|+\beta_t(x)\,v^2\) with \(\alpha_t,\beta_t\ge0\), then \(\partial_v f^{(a,b)}_t(x,0)=[-\alpha_t(x),\alpha_t(x)]\) and
\[
B_t=\{\,x\in\R:\ |h_t(x)|\le \alpha_t(x)\,\}.
\]
The width of the band therefore adjusts with liquidity/impact heterogeneity across \(x\). In the \(x\)–independent linear–quadratic model \(f_{\alpha,\beta}(v)=\alpha|v|+\beta v^2\), the subdifferential is \([-\alpha,\alpha]\), yielding \(B_t=\{\,|h_t|\le\alpha\,\}\); off the band, \eqref{eq:conjugacy-KKT} gives the explicit displacement rule \eqref{eq:relation-h-y-KKT}.

Band boundaries are free boundaries determined by the contact (equal–slope) system as their location governs where transport is active and where the coupling is the identity, thereby shaping both pricing functionals and hedging policies. Band geometry yields sharp comparative statics in the liquidity parameters: increasing a linear cost widens \(B_t\), while increasing a quadratic penalty shrinks off–band displacements (see \eqref{eq:conjugacy-KKT} and \eqref{eq:FOC_expanded}). Regularity and monotonicity of the band further underpin numerical stability and convergence, including the vanishing–friction regime of Corollary~\ref{cor:vanish}.

From a purely financial viewpoint, trade bands formalize threshold behaviour typical of market makers and high–frequency traders. Small price moves are not traded because expected gains do not offset costs, whereas sufficiently large moves trigger immediate rebalancing along the optimal up/down maps. This has direct implications for robust pricing of path–dependent claims. For lookbacks, wider bands dampen tail–pushing; for barriers, bands near the barrier reduce effective crossing under worst–case transport; for Asians, bands suppress frequent small rebalances and reduce short–horizon dispersion. The remainder of this section derives the dual/KKT system, specializes it to linear–quadratic frictions to obtain explicit off–band maps and band boundaries, and discusses how these features propagate to the pricing of such path–dependent contracts.

\subsection{Dual shift and KKT system}
\label{subsec:KKT}

Here we present the complementary–slackness and Fenchel–Young subgradient relations and derive the displacement rule (particulary for the linear–quadratic case).


\begin{lemma}
\label{lem:dual_shift_KKT}
Assume $\varphi,\psi:\R\to\R$ are Borel functions satisfying
\[
\varphi\in L^1(\mu), \qquad \psi\in L^1(\eta),
\]
and that $\varphi$ is convex and $\psi$ is concave.
Let $h:\R\to\R$ be Borel such that $x\mapsto f^{(a,b)\,*}_t(x,h(x))\in L^1(\mu)$.
Consider the one time–step primal problem \eqref{eq:onestep_objective_adjusted}. Define the shifted potentials
\[
\widehat{\varphi}(x):=\varphi(x)+\mathsf{V}_{t+1}(x),
\qquad
\widehat{\psi}(y):=\psi(y)-\mathsf{V}_{t+1}(y).
\]
Then the dual feasibility constraint for \eqref{eq:onestep_objective_adjusted} leads to
\[
\varphi(x)+\psi(y)+h(x)\,(y-x)\ \le\ \widetilde c_t(x,y)\qquad\forall x,y,
\]
which is equivalent to
\[
\widehat{\varphi}(x)+\widehat{\psi}(y)+h(x)\,(y-x)\ \le\ -\,f^{(a,b)}_t(x,y-x)\qquad\forall x,y.
\]
In particular, after the sign change $(\phi,\psi,h):=(-\widehat{\varphi},-\widehat{\psi},-h)$,
one obtains the standard one time–step martingale dual with pure friction cost,
\begin{equation}
\label{eq:dual-1step-KKT}
\sup_{\phi,\psi,h}\ \biggl\{
\int \phi\,\mathrm{d}\mu + \int \psi\,\mathrm{d}\eta:
\ \phi(x)+\psi(y)+h(x)\,(y-x)\ \le\ f^{(a,b)}_t(x,y-x)\ \ \forall x,y
\biggr\},
\end{equation}
and the dual objectives differ by the constant $\eta(\mathsf V_{t+1})-\mu(\mathsf V_{t+1})$:
\[
\int \varphi\,\mathrm d\mu+\int \psi\,\mathrm d\eta
\;=\;\Big(\int \phi\,\mathrm d\mu+\int \psi\,\mathrm d\eta\Big)\;+\;\eta(\mathsf V_{t+1})-\mu(\mathsf V_{t+1}).
\]
\end{lemma}

\begin{proof}
The proof is directly evident. Consider $\widetilde c_t$ in the constraint
$\varphi+\psi+h(\cdot)(y-x)\le \mathsf V_{t+1}(y)-\mathsf V_{t+1}(x)-f^{(a,b)}_t(x,y-x)$
and apply rearrangement to obtain
$\widehat{\varphi}+\widehat{\psi}+h(\cdot)(y-x)\le -f^{(a,b)}_t(x,y-x)$.
Setting $(\phi,\psi,h):=(-\widehat{\varphi},-\widehat{\psi},-h)$ yields the constraint in
\eqref{eq:dual-1step-KKT}. The stated objective shift follows from
$\int \widehat{\varphi}\,\mathrm d\mu=\int \varphi\,\mathrm d\mu+\mu(\mathsf V_{t+1})$ and
$\int \widehat{\psi}\,\mathrm d\eta=\int \psi\,\mathrm d\eta-\eta(\mathsf V_{t+1})$.
\end{proof}

Let $\pi^\ast$ be optimal for \eqref{eq:onestep_objective_adjusted} and
$(\phi,\psi,h)$ be optimal for \eqref{eq:dual-1step-KKT}.
Complementary slackness yields, for $\pi^\ast$–a.e.\ $(x,y)$,
\begin{equation}
\label{eq:slack-equality-KKT}
\phi(x)+\psi(y)+h(x)\,(y-x) \;=\; f^{(a,b)}_t    \left(x,\,y-x\right).
\end{equation}
By Fenchel–Young conjugacy in the trade increment $v:=y-x$,
\begin{equation}
\label{eq:conjugacy-KKT}
y-x \ \in\ \partial f^{(a,b)\,*}_t    \left(x,\,h(x)\right)
\qquad\Longleftrightarrow\qquad
h(x)\ \in\ \partial_v f^{(a,b)}_t    \left(x,\,y-x\right).
\end{equation}
We refer to $h_t:=h$ as the dual slope where it determines the trade band
via the condition $h_t(x)\in\partial_v f^{(a,b)}_t(x,0)$.

In the linear-quadratic case, For $f_{\alpha,\beta}(v)=\alpha|v|+\beta v^2$ (no $x$–dependence),
\[
\partial f_{\alpha,\beta}(v)=
\begin{cases}
\{\alpha\,\mathrm{sgn}(v)+2\beta v\}, & v\neq 0,\\[2pt]
[-\alpha,\alpha], & v=0,
\end{cases}
\]
and \eqref{eq:conjugacy-KKT} yields the displacement rule
\begin{equation}
\label{eq:relation-h-y-KKT}
\begin{cases}
|h(x)|<\alpha \ \Longrightarrow\ y=x,\\[4pt]
|h(x)|>\alpha \ \Longrightarrow\ y-x=\dfrac{h(x)-\alpha\,\mathrm{sgn}   \big(h(x)\big)}{2\beta}, \quad (\beta>0),\\[8pt]
|h(x)|=\alpha \ \Longrightarrow\ y-x\in\{0\}\ \text{or any limit consistent with }(\mu,\eta).
\end{cases}
\end{equation}
If $\beta=0$, the second line is absent, the band $|h(x)|\le \alpha$ enforces $y=x$ on its interior.

When $f^{(a,b)}_t$ is differentiable in $v$ at $v=y-x$ and $\psi$ is differentiable at $y$,
differentiating \eqref{eq:slack-equality-KKT} with respect to $y$ yields
\begin{equation}
\label{eq:psi-slope-KKT}
\psi'(y)\;+\;h(x)\;=\;\partial_v f^{(a,b)}_t   \left(x,\,y-x\right).
\end{equation}
In particular, for $f_{\alpha,\beta}$ and $y\neq x$,
\[
\psi'(y)=\alpha\,\mathrm{sgn}(y-x)+2\beta\,(y-x)-h(x),
\]
whereas at $y=x$ the relation holds in the sense of subgradients.

\begin{remark}[Connection to the global dual \eqref{eq:dual_problem}]
\label{rmk:local-to-global-dual}
The pathwise dual in \cref{subsec:dual_formulation} is written in terms of global, time–indexed
variables $(u_t,\Delta_t)_{t=0}^{N}$ via the functional $\Psi_{u,\Delta}$ in
\eqref{eq:Psi_dual_functional} and the pointwise superhedge constraint
\eqref{eq:dual_inequality}. By contrast, the one time–step dual in
\eqref{eq:dual-1step-KKT} is formulated at a fixed time $t$ with local variables
$(\varphi_t,\psi_t,h_t)$ after the dual shift of Lemma~\ref{lem:dual_shift_KKT}.
These two formulations are canonically equivalent once one sums the one–step
inequalities over $t$ and lets the continuation values $\mathsf V_{t+1}$ telescope.

More precisely, for each $t$ the shifted one tim–step constraint reads
\[
\widehat{\varphi}_t(x)+\widehat{\psi}_t(y)+h_t(x)\,(y-x)\ \le\ -\,f^{(a,b)}_t(x,y-x),
\qquad
\widehat{\varphi}_t:=\varphi_t+\mathsf V_{t+1},\quad
\widehat{\psi}_t:=\psi_t-\mathsf V_{t+1}.
\]
Evaluating at $(x,y)=(S_t,S_{t+1})$ and summing in $t$ yields the telescoping identity
\[
\sum_{t=0}^{N-1}\!\Big\{\varphi_t(S_t)+\psi_t(S_{t+1})
+h_t(S_t)\big(S_{t+1}-S_t\big)\Big\}
\;\le\;
-\,\sum_{t=0}^{N-1}\!f^{(a,b)}_t\big(S_t,S_{t+1}-S_t\big)
+\sum_{t=0}^{N-1}\!\Big\{\mathsf V_{t+1}(S_t)-\mathsf V_{t+1}(S_{t+1})\Big\}.
\]
Defining the global dual variables by the (non-unique) identification
\[
\Delta_t:=h_t,\qquad
u_t:=\varphi_t-\big(\mathsf V_t-\mathsf V_{t+1}\big)\ \ (t=0,\dots,N-1),\qquad
u_N:=\lambda,\ \ \nu(\lambda):=\mu_N(\lambda),
\]
and using Fenchel--Young, $\,\Delta_t(S_{t+1}-S_t)-f^{(a,b)\,*}_t(S_t,\Delta_t)\le
-\,f^{(a,b)}_t(S_t,S_{t+1}-S_t)$, we recover the global pathwise inequality
\eqref{eq:dual_inequality} with $\Psi_{u,\Delta}$ as in \eqref{eq:Psi_dual_functional}.
Conversely, given $(u_t,\Delta_t)$ satisfying \eqref{eq:dual_inequality}, one can set
$h_t:=\Delta_t$ and choose $(\varphi_t,\psi_t)$ so that the one–step constraint
\eqref{eq:dual-1step-KKT} holds together with the shift of Lemma~\ref{lem:dual_shift_KKT}. 
The freedom in $(\varphi_t,\psi_t)$ corresponds to adding telescope–neutral terms.

Thus, the global dual \eqref{eq:dual_problem} is exactly the time–aggregation of the
one–step duals \eqref{eq:dual-1step-KKT} after the dual shift. The potentials
$u_t$ encode the static (semi–static) legs once the continuation values are absorbed,
while $\Delta_t$ is the predictable trading strategy. 
\end{remark}

\subsection{Geometry of trade bands}
\label{subsec:band_geometry}

Let $h_t$ be an optimal dual slope for the one time–step dual after the shift in \eqref{eq:dual-1step-KKT}.
Define the trade band
\begin{equation}
\label{eq:band_def_repeat}
B_t:=\bigl\{x\in\R:\ h_t(x)\in \partial_v f^{(a,b)}_t(x,0)\bigr\}.
\end{equation}
In the linear–quadratic case $f_{\alpha,\beta}(v)=\alpha|v|+\beta v^2$ (no $x$–dependence), $\partial_v f_{\alpha,\beta}(0)=[-\alpha,\alpha]$, hence $B_t=\{\,|h_t|\le\alpha\,\}$.

We work on a fixed irreducible component $I$ of $\{\Delta F_t>0\}$ under the hypotheses of \Cref{thm:fric_monotone}. On $I$ there exist Borel maps
\[
T_d^{(t)}\le \mathrm{Id}\le T_u^{(t)},\qquad
T_d^{(t)}\ \text{nonincreasing},\quad T_u^{(t)}\ \text{nondecreasing},
\]
and weights $\theta_t\in[0,1]$ such that the conditional law is bi–atomic as in \eqref{eq:selection_system}. We take the right–continuous versions of $T_d^{(t)}$ and $T_u^{(t)}$.

\begin{lemma}
\label{lem:band_identity}
On $I$, for $\mu$–a.e.\ $x$,
\[
x\in B_t\quad\Longleftrightarrow\quad T_d^{(t)}(x)=T_u^{(t)}(x)=x.
\]
In particular,
\[
B_t\cap I=\bigl\{x\in I:\ T_u^{(t)}(x)=x\bigr\}
=\bigl\{x\in I:\ T_d^{(t)}(x)=x\bigr\}.
\]
\end{lemma}

\begin{proof}
Suppose $x\in B_t$, so $h_t(x)\in\partial_v f^{(a,b)}_t(x,0)$. By the pointwise Fenchel–Young optimality equality for the dual pair \((f^{(a,b)}_t,\,f^{(a,b)\,*}_t)\) (refer \cref{subsec:KKT}), together with \eqref{eq:slack-equality-KKT}–\eqref{eq:conjugacy-KKT}, any support point $(x,y)$ of $\pi^\star$ must satisfy $h_t(x)\in\partial_v f^{(a,b)}_t\bigl(x,y-x\bigr)$. For fixed $x$, the subdifferential map $v\mapsto\partial_v f^{(a,b)}_t(x,v)$ is monotone. If $v\mapsto f^{(a,b)}_t(x,v)$ is strictly convex in a neighborhood of $0$, then the subgradient is single–valued there and we deduce $y-x=0$, hence $\pi^\star(\cdot\mid x)=\delta_x$ and $T_d^{(t)}(x)=T_u^{(t)}(x)=x$. In the proportional (purely linear) regime in $v$, the subgradient at $0$ is an interval. Equality in Fenchel–Young at $v=0$ shows that $(x,x)$ is a contact point, and choosing this canonical selection on $B_t$ again yields $\pi^\star(\cdot\mid x)=\delta_x$ and $T_d^{(t)}(x)=T_u^{(t)}(x)=x$.

Conversely, if $T_u^{(t)}(x)=T_d^{(t)}(x)=x$, then $(x,x)$ belongs to the contact set of the dual inequality. Evaluating the Fenchel–Young equality at $v=0$ gives $h_t(x)\in\partial_v f^{(a,b)}_t(x,0)$, hence $x\in B_t$.
\end{proof}

\begin{prop}[Structure of $B_t$ and off–band maps]
\label{prop:band_structure}
Let $I$ be an irreducible component of $\{\Delta F_t>0\}$ and assume the hypotheses of Theorem~\ref{thm:fric_monotone}. Then:
\begin{enumerate}
\item $B_t\cap I$ is a Borel subset of $I$, and for every $x\in B_t$ one has $T_d^{(t)}(x)=T_u^{(t)}(x)=x$. If, in addition, the monotone representatives $T_d^{(t)},T_u^{(t)}$ are taken right–continuous (as we do) or $x\mapsto h_t(x)$ is continuous, then $B_t\cap I$ is relatively closed in $I$.
\item $I\setminus B_t$ is a countable union of open intervals. On each such open interval $J\subset I\setminus B_t$, the maps satisfy $T_d^{(t)}(x)<x<T_u^{(t)}(x)$ for all $x\in J$, with $T_d^{(t)}$ strictly decreasing and $T_u^{(t)}$ strictly increasing on $J$.

In the linear–quadratic case $f_{\alpha,\beta}(v)=\alpha|v|+\beta v^2$ with $\beta>0$ (no $x$–dependence),
\begin{equation}
\label{eq:outside_affine_repeat}
T_{u/d}^{(t)}(x)
\;=\;x+\frac{h_t(x)-\alpha\,\mathrm{sgn}\big(h_t(x)\big)}{2\beta},
\qquad x\notin B_t,
\end{equation}
which is the specialization of the KKT relation \eqref{eq:relation-h-y-KKT}.
\end{enumerate}
\end{prop}

\begin{proof}
By Lemma~\ref{lem:band_identity}, $B_t=\{x:\,T_u^{(t)}(x)=T_d^{(t)}(x)=x\}$ on $I$. Since $T_{u/d}^{(t)}$ are Borel and monotone, the set $\{T_u^{(t)}=x\}$ is Borel in $I$, the same holds for $T_d^{(t)}$. If we take the right–continuous representatives of the monotone maps, then $x\mapsto T_u^{(t)}(x)-x$ is right–continuous on $I$ and $\{T_u^{(t)}=x\}$ is relatively closed, the same applies to $T_d^{(t)}$. This proves the first item.

For the second item, the complement of a relatively closed subset of an open interval is a countable union of open intervals. Fix an open $J\subset I\setminus B_t$. By Lemma~\ref{lem:band_identity}, $T_u^{(t)}(x)>x$ and $T_d^{(t)}(x)<x$ for all $x\in J$. Suppose $T_u^{(t)}$ were not strictly increasing on $J$. Then there exist $x_1<x_2$ in $J$ with $T_u^{(t)}(x_1)=T_u^{(t)}(x_2)=:y^\ast$. Using the bi–atomic structure (Theorem~\ref{thm:fric_monotone}) at $x_1,x_2$ and the fact that $T_d^{(t)}$ is nonincreasing, the four points
\[
(x_1,T_d^{(t)}(x_1)),\ (x_1,y^\ast),\ (x_2,T_d^{(t)}(x_2)),\ (x_2,y^\ast)
\]
form a crossing rectangle, which contradicts left–monotonicity of the support (Theorem~\ref{thm:fric_monotone}). Hence $T_u^{(t)}$ is strictly increasing on $J$. The argument for $T_d^{(t)}$ is analogous. The formula \eqref{eq:outside_affine_repeat} in the linear quadratic case is the KKT displacement rule \eqref{eq:relation-h-y-KKT}.
\end{proof}

\begin{prop}[Comparative statics in $(\alpha,\beta)$ in the linear–quadratic case]
\label{prop:comparative_statics}
Assume $f_{\alpha,\beta}(v)=\alpha|v|+\beta v^2$ with $\beta>0$. Fix an irreducible component $I$ of $\{\Delta F_t>0\}$. Suppose that:

\begin{enumerate}
\item The measures $\mu$ and $\eta$ are absolutely continuous on the relevant ranges with continuous densities $f_\mu,f_\eta$ that are strictly positive there.
\item On each open interval $J\subset I\setminus B_t$, the selection system consisting of the mass identity and the equal–slope condition admits a unique solution $(T_d^{(t)},T_u^{(t)},\theta_t)$ which depends continuously on $(\alpha,\beta)$, and $\mathsf V_{t+1}''$ is continuous and bounded on the images $T_d^{(t)}(J)$ and $T_u^{(t)}(J)$.
\item (Sign and curvature conditions) Along $J$ one has
\[
\Delta F_t'(T_d^{(t)}(x))<0
\quad\text{and}\quad
f_\eta\bigl(T_u^{(t)}(x)\bigr)>0
\qquad\text{for a.e.\ }x\in J,
\]
and the curvature dominance
\[
2\beta>\sup\{\mathsf V_{t+1}''(y):\, y\in T_d^{(t)}(J)\cup T_u^{(t)}(J)\}.
\]
\end{enumerate}
Then:
\begin{enumerate}
\item The band expands weakly in each parameter: if $\alpha_2\ge\alpha_1$, then $B_t(\alpha_1,\beta)\cap I\subseteq B_t(\alpha_2,\beta)\cap I$, and if $\beta_2\ge\beta_1$, then $B_t(\alpha,\beta_1)\cap I\subseteq B_t(\alpha,\beta_2)\cap I$.
\item On each $J\subset I\setminus B_t$, the off–band displacements decrease with either parameter:
\[
\frac{\partial}{\partial \alpha}\,|T_{u/d}^{(t)}(x)-x|\ \le\ 0,
\qquad
\frac{\partial}{\partial \beta}\,|T_{u/d}^{(t)}(x)-x|\ \le\ 0
\quad\text{for a.e.\ }x\in J.
\]
\end{enumerate}
\end{prop}

\begin{proof}
Work on a fixed open $J\subset I\setminus B_t$ and write $T_{d/u}$ for $T_{d/u}^{(t)}$. By Theorem~\ref{thm:fric_monotone} and the potential formalism, $(T_d,T_u)$ satisfy on $J$ the mass identity
\begin{equation}
\label{eq:CS-mass}
F_\eta\bigl(T_u(x)\bigr)-F_\mu(x)=\Delta F_t\bigl(T_d(x)\bigr),
\end{equation}
and the linear-quadratic equal–slope condition
\begin{equation}
\label{eq:CS-slope}
\mathsf V_{t+1}'\bigl(T_u(x)\bigr)-\mathsf V_{t+1}'\bigl(T_d(x)\bigr)
=2\alpha+2\beta\bigl(T_u(x)-T_d(x)\bigr).
\end{equation}
Set $z=(T_d,T_u)$ and define
\[
H_1(x,z):=F_\eta(T_u)-F_\mu(x)-\Delta F_t(T_d),\qquad
H_2(z;\alpha,\beta):=\mathsf V_{t+1}'(T_u)-\mathsf V_{t+1}'(T_d)-2\alpha-2\beta\,(T_u-T_d).
\]
By the hypotheses and the chain rule, for a.e.\ $x\in J$,
\[
\partial_{T_d}H_1= -\Delta F_t'(T_d)=f_\mu(T_d)-f_\eta(T_d)=:a_d>0,\qquad
\partial_{T_u}H_1= f_\eta(T_u)=:a_u>0,
\]
and
\[
\partial_{T_d}H_2= -\mathsf V_{t+1}''(T_d)+2\beta,\qquad
\partial_{T_u}H_2= \mathsf V_{t+1}''(T_u)-2\beta.
\]
Hence the Jacobian
\[
J(x):=\frac{\partial(H_1,H_2)}{\partial(T_d,T_u)}
=\begin{pmatrix}
a_d & a_u\\[2pt]
2\beta-\mathsf V_{t+1}''(T_d) & \mathsf V_{t+1}''(T_u)-2\beta
\end{pmatrix}
\]
is invertible on $J$, with
\[
\det J(x)=a_d\bigl(\mathsf V_{t+1}''(T_u)-2\beta\bigr)-a_u\bigl(2\beta-\mathsf V_{t+1}''(T_d)\bigr)<0,
\]
by the curvature dominance. Differentiating the system $(H_1,H_2)=(0,0)$ with respect to $\alpha$ (keeping $x$ fixed) yields
\[
J(x)\begin{pmatrix}\partial_\alpha T_d\\ \partial_\alpha T_u\end{pmatrix}=\begin{pmatrix}0\\ 2\end{pmatrix}.
\]
Cramer’s rule gives
\[
\partial_\alpha T_d=\frac{-\,2\,a_u}{\det J}\ \ge\ 0,\qquad
\partial_\alpha T_u=\frac{2\,a_d}{\det J}\ \le\ 0,
\]
so both displacements $x-T_d$ and $T_u-x$ are weakly decreasing in $\alpha$. The same computation with respect to $\beta$ gives
\[
J(x)\begin{pmatrix}\partial_\beta T_d\\ \partial_\beta T_u\end{pmatrix}=\begin{pmatrix}0\\ 2\,(T_u-T_d)\end{pmatrix}
\quad\Longrightarrow\quad
\partial_\beta T_d=\frac{-\,2\,(T_u-T_d)\,a_u}{\det J}\ \ge\ 0,\qquad
\partial_\beta T_u=\frac{2\,(T_u-T_d)\,a_d}{\det J}\ \le\ 0,
\]
and the same monotonicity in $\beta$ follows. Finally, $B_t\cap I=\{x:\ T_u(x)=x\}=\{x:\ T_d(x)=x\}$ by \cref{lem:band_identity}, since both $T_u-x$ and $x-T_d$ decrease pointwise with either parameter on $J$ and vanish on $B_t$, the zero set expands weakly with $\alpha$ or $\beta$, proving band expansion.
\end{proof}

\begin{remark}[Quantitative identities on and off the band]
\label{rem:band_quantitative}
On $B_t$ the optimizer keeps mass on the diagonal, so $v=y-x=0$ and the adjusted one time–step contribution equals $-\,f^{(a,b)}_t(x,0)$, which is $0$ under the standard normalization $f^{(a,b)}_t(x,0)=0$. On $I\setminus B_t$, the contribution at $x$ is
\[
\theta_t(x)\Bigl(\mathsf{V}_{t+1}\bigl(T_u^{(t)}(x)\bigr)-\mathsf{V}_{t+1}(x)-f^{(a,b)}_t\bigl(x,\,T_u^{(t)}(x)-x\bigr)\Bigr)
+\bigl(1-\theta_t(x)\bigr)\Bigl(\mathsf{V}_{t+1}\bigl(T_d^{(t)}(x)\bigr)-\mathsf{V}_{t+1}(x)-f^{(a,b)}_t\bigl(x,\,T_d^{(t)}(x)-x\bigr)\Bigr),
\]
with $\theta_t(x)=\dfrac{x-T_d^{(t)}(x)}{T_u^{(t)}(x)-T_d^{(t)}(x)}$ as in \eqref{eq:selection_system}. The endpoints $T_{u/d}^{(t)}$ are determined by the mass relations \eqref{eq:mass_conservation_fric} and the equal–slope condition \eqref{eq:FOC_expanded}. In the linear-quadratic case one can further use \eqref{eq:outside_affine_repeat} to bound or evaluate the off–band terms explicitly.
\end{remark}

\begin{figure}[t]
  \centering
  \includegraphics[width=0.7\textwidth,trim=0.2cm 0.2cm 0.2cm 1cm,clip=true]{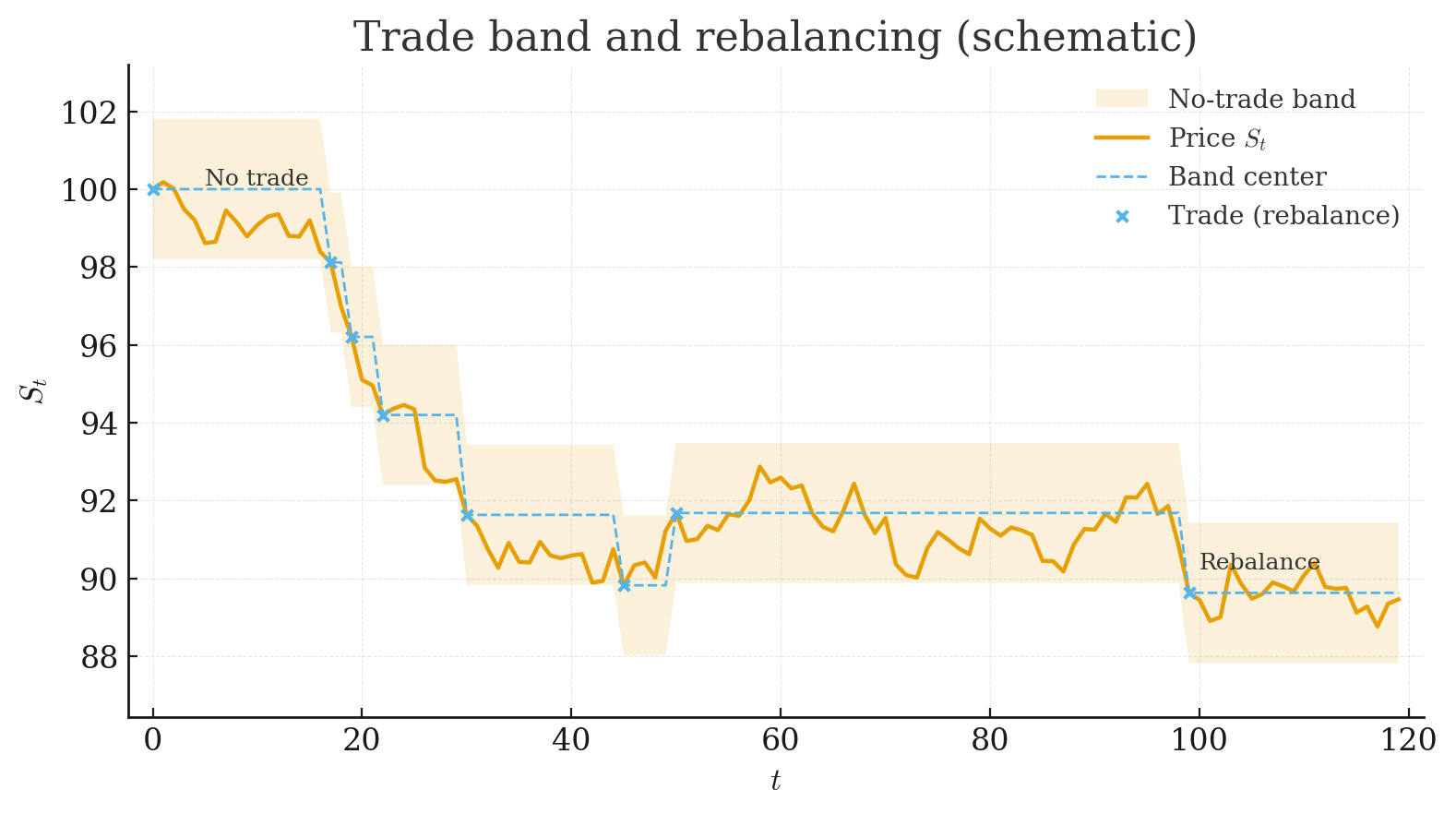}
  \vspace{4pt}
\caption{Schematic illustration of the no–trade region and optimal rebalancing under frictions. The orange line shows the price process \(S_t\) evolving within the state–dependent trade band \(B_t\) defined by \eqref{eq:band_def_repeat}.
    While \(S_t\) remains inside \(B_t\), the optimal martingale coupling coincides with the identity—no rebalancing occurs and the portfolio position is held fixed.
    When \(S_t\) exits the band, a trade is triggered (blue markers), corresponding to an update along the monotone graphs \(T_d^{(t)}\) and \(T_u^{(t)}\) described in \cref{thm:fric_monotone}.
    The band width reflects the local liquidity or impact cost parameters \((\alpha_t,\beta_t)\). Widening \(\alpha_t\) increases the inaction region, whereas larger \(\beta_t\) reduces off–band displacements.
    Financially, \(B_t\) captures the threshold behaviour of market makers: small price fluctuations are absorbed without trading, while sufficiently large moves trigger optimal rebalancing.}
\label{fig:trade-band}
\end{figure}

\subsection{Impact on high-frequency trading}
\label{subsec:financial_impact}

We next quantify the operational impact that are directly relevant to high-frequency execution and market-making.

\begin{prop}[Turnover representation and bounds]
\label{prop:turnover}
Let $I$ be an irreducible component of $\{\Delta F_t>0\}$, and let $\pi^\star_t$ be the frictional left–monotone optimizer with binomial kernel \eqref{eq:selection_system}. The one time–step (absolute) turnover is
\[
\mathcal{T}_t
:=\mathbb{E}_{\mu}\Big[\mathbb{E}\big[\,|Y-X|\,\big|\,X\,\big]\Big]
= \int_{I\setminus B_t}  \Big(\theta_t(x)\,\big|T_u^{(t)}(x)-x\big|+\big(1-\theta_t(x)\big)\,\big|x-T_d^{(t)}(x)\big|\Big)\,\mathrm d\mu(x).
\]
Moreover, on $I\setminus B_t$,
\begin{equation}
\label{eq:turnover_pointwise}
\mathbb{E}\big[\,|Y-X|\,\big|\,X=x\,\big]
=2\,\theta_t(x)\big(1-\theta_t(x)\big)\,\Big(T_u^{(t)}(x)-T_d^{(t)}(x)\Big)
\ \le\ \frac{1}{2}\,\Big(T_u^{(t)}(x)-T_d^{(t)}(x)\Big),
\end{equation}
and consequently
\begin{equation}
\label{eq:turnover_bound}
\mathcal{T}_t\ \le\ \frac{1}{2}\int_{I\setminus B_t}  \Big(T_u^{(t)}(x)-T_d^{(t)}(x)\Big)\,\mathrm d\mu(x).
\end{equation}
In the linear–quadratic case with $\beta>0$ and $f_{\alpha,\beta}(v)=\alpha|v|+\beta v^2$,
\[
\big|T_{u/d}^{(t)}(x)-x\big|=\frac{\big(|h_t(x)|-\alpha\big)_+}{2\beta},
\]
hence
\begin{equation}
\label{eq:turnover_LQ_bound}
\mathbb{E}\big[\,|Y-X|\,\big|\,X=x\,\big]
\ \le\ \frac{\big(|h_t(x)|-\alpha\big)_+}{2\beta},
\qquad
\mathcal{T}_t\ \le\ \frac{1}{2\beta}\int_{I}\big(|h_t(x)|-\alpha\big)_+\,\mathrm d\mu(x),
\end{equation}
where $(\cdot)_+=\max\{\cdot,0\}$.
\end{prop}

\begin{proof}
Using \eqref{eq:selection_system}, conditional on $X=x$,
\[
\mathbb{E}\big[\,|Y-X|\,\big|\,X=x\,\big]
=\theta_t(x)\,\big|T_u^{(t)}(x)-x\big|+\big(1-\theta_t(x)\big)\,\big|x-T_d^{(t)}(x)\big|.
\]
On $B_t$ we have $T_d^{(t)}(x)=T_u^{(t)}(x)=x$, so the conditional mean vanishes. On $I\setminus B_t$, $T_d^{(t)}(x)\le x\le T_u^{(t)}(x)$ and the martingale constraint gives
\[
\theta_t(x)\big(T_u^{(t)}(x)-x\big)=(1-\theta_t(x))\big(x-T_d^{(t)}(x)\big)=:\mathfrak m(x).
\]
Hence the conditional mean equals $2\,\mathfrak m(x)$, and using
$T_u^{(t)}(x)-x=(1-\theta_t(x))\big(T_u^{(t)}(x)-T_d^{(t)}(x)\big)$ and
$x-T_d^{(t)}(x)=\theta_t(x)\big(T_u^{(t)}(x)-T_d^{(t)}(x)\big)$ yields \eqref{eq:turnover_pointwise}. Since $\theta_t(1-\theta_t)\le1/4$, the inequality in \eqref{eq:turnover_pointwise} follows. Integrating over $I\setminus B_t$ gives \eqref{eq:turnover_bound}. In the linear-quadratic case, the displacement identity is \eqref{eq:relation-h-y-KKT}, which implies \eqref{eq:turnover_LQ_bound}.
\end{proof}

\begin{prop}[Comparative statics of turnover]
\label{prop:turnover_CS}
Assume $v\mapsto f^{(a,b)}_t(x,v)$ is $C^2$ in $v$, strictly convex with $\partial_{vv}f^{(a,b)}_t(x,v)\ge m>0$, and that the equal–slope system \eqref{eq:FOC_expanded} is nondegenerate on $I\setminus B_t$. Then the directional derivatives of $\mathcal{T}_t$ with respect to the friction parameters satisfy
\[
D_\alpha \mathcal{T}_t \le 0,
\qquad
D_\beta \mathcal{T}_t \le 0.
\]
In particular, increasing the half–spread $\alpha$ or the quadratic penalty $\beta$ weakly reduces optimal one time–step turnover.
\end{prop}

\begin{proof}
Under the stated regularity, the implicit function theorem applies to \eqref{eq:FOC_expanded}, giving differentiable dependence of $x\mapsto (T_d^{(t)}(x),T_u^{(t)}(x),\theta_t(x))$ on $(\alpha,\beta)$ off the band. Proposition~\ref{prop:comparative_statics} shows that $B_t$ expands weakly as either parameter increases. Differentiating the representation of $\mathcal{T}_t$ in Proposition~\ref{prop:turnover} and using the barycenter identity shows the off–band displacement terms decrease with either parameter. Together with the expansion of $B_t$, this yields the stated signs.
\end{proof}

\begin{prop}[Execution cost and price impact]
\label{prop:exec_cost}
Let $\mathcal{K}_t:=\mathbb{E}\big[f^{(a,b)}_t \big(X,Y-X\big)\big]$ be the one time–step expected trading cost under $\pi^\star_t$. Then $\mathcal{K}_t=0$ on $B_t$ and, on $I\setminus B_t$,
\[
\mathcal{K}_t
=\int_{I\setminus B_t}  \Big(\theta_t(x)\,f^{(a,b)}_t \big(x,T_u^{(t)}(x)-x\big)
+\big(1-\theta_t(x)\big)\,f^{(a,b)}_t \big(x,T_d^{(t)}(x)-x\big)\Big)\,\mathrm d\mu(x).
\]
In the linear–quadratic case with $\beta>0$,
\[
\mathcal{K}_t
\ \ge\
\alpha\,\mathcal{T}_t,
\qquad
\mathcal{K}_t
\ \ge\
\beta \int_{I\setminus B_t}  \theta_t(x)\big(1-\theta_t(x)\big)\Big(T_u^{(t)}(x)-T_d^{(t)}(x)\Big)^2\,\mathrm d\mu(x).
\]
\end{prop}

\begin{proof}
Disintegrate $\pi^\star_t$ as in \eqref{eq:selection_system}. On the band, $Y-X=0$, hence zero cost. For the linear-quadratic setting, $f_{\alpha,\beta}(v)=\alpha|v|+\beta v^2$, we have
\[
\mathbb{E}\big[\,|Y-X|\,\big|\,X=x\,\big]=2\,\theta_t(1-\theta_t)\,\Delta(x),
\qquad
\mathbb{E}\big[\, (Y-X)^2 \,\big|\,X=x\,\big]=\theta_t(1-\theta_t)\,\Delta(x)^2,
\]
where $\Delta(x):=T_u^{(t)}(x)-T_d^{(t)}(x)$. Taking expectations gives the stated lower bounds with coefficients $\alpha$ and $\beta$.
\end{proof}

\begin{remark}
The quantity $\mathcal{T}_t$ is the expected absolute inventory change per rehedge. $\mathcal{K}_t$ is the expected one time–step trading cost. Proposition~\ref{prop:turnover_CS} shows that increasing frictions reduces endogenous turnover (wider no–trade zones), while Proposition~\ref{prop:exec_cost} quantifies how much execution budget is consumed off–band. In particular, \eqref{eq:turnover_LQ_bound} identifies a spread $(|h_t|-\alpha)_+/(2\beta)$ translating the dual slope into an effective displacement scale. These statistics can be aggregated over $t$ to obtain expected total turnover and cost profiles for hedging programs under frictions, and they motivate calibration of $(\alpha,\beta)$ to realized spread/impact data in market-making applications.
\end{remark}

\subsection{Example payoffs: lookback, barrier, and Asian claims}
\label{subsec:worked_payoffs}

In path–dependent problems the continuation value at $t{+}1$ depends on both the next price $y=S_{t+1}$ and a low–dimensional state $ \xi_{t+1}$ summarizing the past. Fixing $ \xi_t$ reduces the one time–step optimization to the statewise setting of Section~\ref{sec:fric_geometry}, with the adjusted cost $\widetilde c_t(x,y; \xi_t)$ of \eqref{eq:adjusted_cost_repeat}. The geometry (monotone two–graph structure, mass identities, and equal–slope selection) then applies for each state, and band behavior reflects the regularity of $y\mapsto \mathsf V_{t+1}(y; \xi_{t+1})$.

\subsubsection*{Lookback (fixed strike)}
Let $\Phi=\big(\max_{0\le s\le N} S_s - K\big)^+$. The sufficient state is the running maximum
\begin{equation}
  \label{eq:lookback_state}
 \xi_t=M_t:=\max\{S_0,\dots,S_t\},\qquad
M_{t+1}= \xi_{t+1}(x,y;M_t)=\max\{M_t,y\}.
\end{equation}
For fixed $M_t=m$, the map $y\mapsto \mathsf V_{t+1}(y;m)$ is nondecreasing, convex, and piecewise $C^1$ with a kink at $y=m$ (the switching locus of the running maximum). Hence $\partial_y^3 \mathsf V_{t+1}(\cdot;m)$ may vanish away from $y=m$, so the strict MSM condition \eqref{eq:MSM_smooth} can fail. By Theorem~\ref{thm:fric_monotone}, the optimizer remains left–monotone with bi–atomic conditionals. Uniqueness may fail as in \cref{rem:degenerate_MSM}. The trade band $B_t=\{x:\ h_t(x)\in\partial_v f^{(a,b)}_t(x,0)\}$ typically widens for states $m$ for which realizations with $y\approx m$ are likely, reflecting inaction near the “at–the–money’’ region of the running maximum.

\subsubsection*{Barrier (up–and–out)}
Let $\Phi=\mathbf 1_{\{\max_{0\le s\le N}S_s<B\}}\cdot \varphi(S_N)$ for bounded Borel $\varphi$. The sufficient state is the barrier flag
\begin{equation}
  \label{eq:barrier_state}
 \xi_t=I_t:=\mathbf 1_{\{\max_{0\le s\le t}S_s<B\}},\qquad
I_{t+1}= \xi_{t+1}(x,y;I_t)= I_t\cdot \mathbf 1_{\{y<B\}}.
\end{equation}
For fixed $I_t=i\in\{0,1\}$, the map $y\mapsto\mathsf V_{t+1}(y;i)$ is piecewise $C^1$ with a kink at $y=B$ when $i=1$, so the frictional MSM may degenerate near the barrier. The equal–slope selection \eqref{eq:FOC_expanded} applies statewise. The optimal endpoints $T_{d/u}^{(t)}(\cdot;i)$ typically jump across the preimage of the barrier, and $B_t$ often abuts the locus where the barrier status switches.

\subsubsection*{Asian (arithmetic average)}
Let $\Phi=\big(\frac{1}{N+1}\sum_{s=0}^N S_s-K\big)^+$. A sufficient state is the running sum
\begin{equation}
  \label{eq:Asian_state}
 \xi_t=A_t:=\sum_{s=0}^t S_s,\qquad
A_{t+1}= \xi_{t+1}(x,y;A_t)=A_t+y.
\end{equation}
For fixed $A_t=a$, the map $y\mapsto\mathsf V_{t+1}(y;a)$ is typically smoother. Under mild regularity of the payoff and of the marginal sequence, one has $y\mapsto\mathsf V_{t+1}(y;a)\in C^3$ with $\partial_y^3\mathsf V_{t+1}(y;a)\ge \kappa_t>0$ on compact sets away from trivial boundary kinks. Hence the strict MSM condition \eqref{eq:MSM_smooth} holds, yielding uniqueness. The bands are thinner, and endpoints vary smoothly with $(\alpha,\beta)$ (Proposition~\ref{prop:comparative_statics}).

For each payoff above, the endpoints are determined by the mass–conservation/coupling identity \eqref{eq:coupling_identity_fric} together with the equal–slope condition \eqref{eq:FOC_expanded}. In the linear–quadratic case $f_{\alpha,\beta}(v)=\alpha|v|+\beta v^2$ with $\beta>0$, the off–band displacement is given by \eqref{eq:relation-h-y-KKT}, and the weights are $\theta_t(x)=\dfrac{x-T_d^{(t)}(x)}{T_u^{(t)}(x)-T_d^{(t)}(x)}$.

\begin{remark}
\label{rem:worked_regularity}
For lookback and barrier claims, kinks in $y\mapsto \mathsf V_{t+1}(\cdot; \xi_{t+1})$ typically preclude \eqref{eq:MSM_smooth} as multiple admissible solutions to \eqref{eq:selection_system} and \eqref{eq:relation-h-y-KKT} may exist, and $B_t$ enlarges near the kink loci (e.g., $y\approx M_t$ or $y\approx B$). For Asian claims, the smoother dependence on $y$ often enforces \eqref{eq:MSM_smooth}, yielding uniqueness and thinner bands. These observations hold statewise and pass to the global optimizer after integrating over the law of $ \xi_t$.
\end{remark}

\section{Multi–marginal extension}
\label{sec:multi_marginal}

We now pass from one time-step to the full $N$ time–step problem. The geometry of $t$–to–$t{+}1$  from \cref{sec:fric_geometry} is applied recursively with the appropriate continuation value, so that the martingale constraint and the monotone structure are enforced at each step. In the frictionless case this timewise program underlies the left–curtain and tilted–curtain constructions of \cite{BJ2016,HLT2016,BeiglbockLabordereTouzi2017}. Here we carry it out in the presence of state–dependent friction \eqref{eq:friction_function}.

\subsection{Recalling the multi–marginal setup}
\label{subsec:mm_setup}

We pass from a single step to a discrete horizon $t=0,\dots,N$ with prescribed one–dimensional marginals $(\mu_t)_{t=0}^N$ in increasing convex order,
\[
\mu_t\preceq_{\mathrm{cx}}\mu_{t+1}\qquad (t=0,\dots,N-1).
\]
On the canonical space $(\R^{N+1},\mathcal B(\R)^{\otimes(N+1)})$ with coordinate process $S_t$, we consider the class of (discrete–time) martingale couplings with these marginals,
\[
\mathcal M(\mu_0,\dots,\mu_N)
:=\Bigl\{\pi\in\mathcal P(\R^{N+1}) \,:\, \pi\circ S_t^{-1}=\mu_t\ \forall t,\ \E_\pi[S_{t+1}\mid S_0,\dots,S_t]=S_t\ \text{a.s.}\Bigr\}.
\]

At each time $t$ we retain the one time–step friction specification $v=y-x\mapsto f^{(a,b)}_t(x,v)$ from \cref{subsec:fric_MSM} and the adjusted cost
\[
\widetilde c_t(x,y)=\mathsf V_{t+1}(y)-\mathsf V_{t+1}(x)-f^{(a,b)}_t(x,y-x).
\]
The rectangle MSM condition \eqref{eq:rectangle_MSM}is assumed timewise. For each $t$ one defines $F_{\mu_t}$, $F_{\mu_{t+1}}$, the potential difference $\Delta F_t=\mathcal C_{\mu_{t+1}}-\mathcal C_{\mu_t}$, and the active set $I_t=\{\Delta F_t>0\}$, which enter the mass–conservation and coupling identities \eqref{eq:mass_conservation_fric}–\eqref{eq:coupling_identity_fric} at time $t$ \cite{BJ2016, HLT2016}.

Under the \cref{def:admissible_payoff}, given a payoff $\Phi:\R^{N+1}\to\R$ , the multi–step frictional martingale optimal transport problem is
\begin{equation}
\label{eq:mm_primal}
\sup_{\pi\in\mathcal M(\mu_0,\dots,\mu_N)}
\E_\pi \left[\ \Phi(S_0,\dots,S_N)\;-\;\sum_{t=0}^{N-1} f^{(a,b)}_t\big(S_t,\,S_{t+1}-S_t\big)\ \right].
\end{equation}
We employ dynamic programming as used in \cite{HLT2016} and define continuation values
\[
\mathsf V_N(y):=\sup_{\pi\in\mathcal M(\mu_N)} \E_\pi \left[\Phi(S_0,\dots,S_{N-1},y)\,\middle|\,S_N=y\right],
\qquad
\mathsf V_t(x):=\sup_{\substack{\pi\in\mathcal M(\mu_t,\dots,\mu_N)}} \E_\pi \left[\Phi\,\middle|\,S_t=x\right],
\]
and at each $t$ solve the one time–step problem between $(\mu_t,\mu_{t+1})$ with adjusted cost $\widetilde c_t$. Under the timewise MSM assumptions and the integrability hypotheses, the one time–step analysis from \cref{thm:fric_monotone} applies directly and seemlessly. Therefore, concatenating the timewise optimal transitions yields an optimizer for \eqref{eq:mm_primal}, and the corresponding dual is obtained by summing the one time–step duals along with the same dual shift as in \cref{subsec:KKT}.

\subsection{Dynamic programming identity and the global dual}

We first record the dynamic programming identity in an integrated form that will be used to
link the one time–step analysis in
\cref{sec:fric_geometry} to the multi–step problem \eqref{eq:mm_primal}. Define the continuation values for $t=N,\dots,0$ by
\[
\mathsf V_N=\Phi,\qquad
\mathsf V_t(x):=\sup_{\substack{\pi\in\mathcal M(\mu_t,\dots,\mu_N)}} 
\mathbb E_\pi \left[\Phi(S_0,\dots,S_N)\,\middle|\,S_t=x\right].
\]
Under the standing integrability assumptions and the timewise MSM condition, the one time–step
value identity holds (cf.\ the frictionless case in \cite{HLT2016}):
\begin{equation}
\label{eq:mm_dpp}
\int_{\R}\mathsf V_t(x)\,\mathrm d\mu_t(x)
\;=\;
\int_{\R}\mathsf V_{t+1}(y)\,\mathrm d\mu_{t+1}(y)
\;+\;
\sup_{\pi_t\in\Pi^M(\mu_t,\mu_{t+1})}\;
\int_{\R^2}\widetilde c_t(x,y)\,\mathrm d\pi_t(x,y),
\qquad t=0,\dots,N-1,
\end{equation}
where
$\widetilde c_t(x,y)=\mathsf V_{t+1}(y)-\mathsf V_{t+1}(x)-f^{(a,b)}_t(x,y-x)$
is the adjusted one time–step cost from \cref{subsec:KKT}.
Identity \eqref{eq:mm_dpp} is obtained by conditioning the global objective at time $t$,
reducing the inner supremum to a one time–step martingale transport between $(\mu_t,\mu_{t+1})$,
and integrating against $\mu_t$, see also \cite{HLT2016,BeiglbockLabordereTouzi2017} for the frictionless derivation
and \S\ref{sec:fric_geometry} for the frictional one time–step reduction.

Next we formulate the global dual as the sum of the one time–step duals following \cref{rmk:local-to-global-dual}. For each $t$ introduce Borel functions
$\phi_t,\psi_t:\R\to\R$ and a predictable multiplier $h_t:\R\to\R$
(depending on $S_t$, i.e.\ on the current abscissa), and impose the
one time–step dual feasibility
\begin{equation}
\label{eq:mm_dual_ineq}
\phi_t(x)+\psi_t(y)+h_t(x)\,(y-x)\ \le\ f^{(a,b)}_t(x,y-x)
\qquad\text{for all }x,y\in\R,\ \ t=0,\dots,N-1.
\end{equation}
The multi–marginal dual problem is then
\begin{equation}
\label{eq:mm_dual}
\sup_{\{\phi_t,\psi_t,h_t\}_{t=0}^{N-1}}
\sum_{t=0}^{N-1}\left(\int_{\R}\phi_t(x)\,\mathrm d\mu_t(x)
+\int_{\R}\psi_t(y)\,\mathrm d\mu_{t+1}(y)\right)
\quad\text{subject to \eqref{eq:mm_dual_ineq} for every }t.
\end{equation}
Equivalently, undoing the dual shift (Lemma~\ref{lem:dual_shift_KKT}),
one may replace $(\phi_t,\psi_t)$ by $(\varphi_t,\chi_t)$ through
$\phi_t=-\widehat\varphi_t$, $\psi_t=-\widehat\chi_t$, so that
$\widehat\varphi_t(x)=\varphi_t(x)+\mathsf V_{t+1}(x)$ and
$\widehat\chi_t(y)=\chi_t(y)-\mathsf V_{t+1}(y)$, and the objective in
\eqref{eq:mm_dual} changes by the constant
$\sum_{t=0}^{N-1}\big(\mu_t(\mathsf V_{t+1})-\mu_{t+1}(\mathsf V_{t+1})\big)$,
which telescopes to $\mu_0(\mathsf V_1)-\mu_N(\Phi)$.

\begin{theorem}[Extension of \cref{thm:fric_monotone} to the multi--marginal setting]
\label{thm:mm_fric_monotone}
Assume the standing integrability and timewise frictional MSM hypotheses. For each $t\in\{0,\dots,N-1\}$ let
\[
\widetilde c_t(x,y)=\mathsf V_{t+1}(y)-\mathsf V_{t+1}(x)-f^{(a,b)}_t(x,y-x),
\qquad
\Delta F_t=\mathcal C_{\mu_{t+1}}-\mathcal C_{\mu_t},\qquad
I_t=\{\Delta F_t>0\}.
\]
Then there exists $\pi^\star\in\mathcal M(\mu_0,\dots,\mu_N)$ such that, for each $t$:
\begin{enumerate}
\item The time-$t$ transition kernel of $\pi^\star$ admits a bi--atomic disintegration on $I_t$:
\[
\pi^\star(\mathrm d\mathbf s)=\mu_t(\mathrm dx)\,K_t^\star(x,\mathrm dy)\,\pi^\star(\mathrm d\mathbf s_{\neq t,t+1}\mid x,y),\qquad
K_t^\star(x,\mathrm dy)=\theta_t(x)\,\delta_{T_u^{(t)}(x)}(\mathrm dy)+\bigl(1-\theta_t(x)\bigr)\,\delta_{T_d^{(t)}(x)}(\mathrm dy),
\]
with Borel maps $T_d^{(t)}\le \mathrm{Id}\le T_u^{(t)}$ satisfying that $T_d^{(t)}$ is nonincreasing and $T_u^{(t)}$ is nondecreasing on $I_t$, and
\[
\theta_t(x)=\frac{x-T_d^{(t)}(x)}{T_u^{(t)}(x)-T_d^{(t)}(x)}\in[0,1]\qquad\text{for }\mu_t\text{--a.e.\ }x\in I_t.
\]
\item The mass identities hold in the following sense:
\[
\mathrm dF_{\mu_{t+1}} \bigl(T_u^{(t)}(x)\bigr)=\theta_t(x)\,\mathrm dF_{\mu_t}(x),
\qquad
\mathrm d\bigl(\Delta F_t\bigr) \bigl(T_d^{(t)}(x)\bigr)=-(1-\theta_t(x))\,\mathrm dF_{\mu_t}(x),
\]
equivalently
\[
F_{\mu_{t+1}} \bigl(T_u^{(t)}(x)\bigr)=F_{\mu_t}(x)+\Delta F_t \bigl(T_d^{(t)}(x)\bigr)
\qquad\text{for }\mu_t\text{--a.e.\ }x\in I_t.
\]
\item The endpoints $(T_d^{(t)},T_u^{(t)})$ satisfy the equal--slope condition of Theorem~\ref{thm:fric_monotone} with continuation $\mathsf V_{t+1}$.
\end{enumerate}
If, in addition, the one time-step optimizer at time $t$ is unique (for example, under strict convexity in $v$ together with the rectangle MSM condition), then $(T_d^{(t)},T_u^{(t)},\theta_t)$ are $\mu_t$--a.e.\ unique, and any optimal $\pi\in\mathcal M(\mu_0,\dots,\mu_N)$ has time-$t$ transition $K_t^\star$.
\end{theorem}

\begin{proof}
Set the continuation values recursively backward by
\[
\mathsf V_N=\Phi,\qquad
\mathsf V_t(x):=\sup_{\pi\in\mathcal M(\mu_t,\dots,\mu_N)}
\mathbb E_\pi \left[\Phi(S_0,\dots,S_N)\,\middle|\,S_t=x\right],
\]
so that the dynamic programming identity \eqref{eq:mm_dpp} holds and decomposes the global value into successive one time–step contributions. In particular, the optimization at time~$t$ reduces to 
\[
\sup_{\pi_t\in\Pi^M(\mu_t,\mu_{t+1})}
\int_{\R^2}\widetilde c_t(x,y)\,\mathrm d\pi_t(x,y),
\qquad
\widetilde c_t(x,y)=\mathsf V_{t+1}(y)-\mathsf V_{t+1}(x)-f^{(a,b)}_t(x,y-x).
\]

By the timewise frictional MSM hypothesis and Theorem~\ref{thm:fric_monotone}, applied with continuation $\mathsf V_{t+1}$, there exists a left–monotone maximizer 
\[
\pi_t^\star\in\Pi^M(\mu_t,\mu_{t+1}),
\]
whose disintegration on $I_t$ is bi–atomic on two monotone graphs 
$T_d^{(t)}\le\mathrm{Id}\le T_u^{(t)}$ with weight $\theta_t$, satisfying
\[
\pi_t^\star(\mathrm dy\mid x)
=\theta_t(x)\,\delta_{T_u^{(t)}(x)}(\mathrm dy)
+\bigl(1-\theta_t(x)\bigr)\,\delta_{T_d^{(t)}(x)}(\mathrm dy),
\]
together with the pushforward relation and coupling identity \eqref{eq:stieltjes_push} and \eqref{eq:coupling_identity_proved}, respectively.

Let $K_t^\star(x,\cdot):=\pi_t^\star(\cdot\mid x)$ and construct the joint measure 
$\pi^\star\in\mathcal P(\R^{N+1})$ inductively using the Ionescu–Tulcea theorem. 
By construction, each $K_t^\star$ pushes $\mu_t$ to $\mu_{t+1}$ and satisfies the martingale constraint 
$\E_{K_t^\star}[Y\mid X=x]=x$. hence 
\[
\pi^\star\in\mathcal M(\mu_0,\dots,\mu_N).
\]

Summing the optimal one–step values in \eqref{eq:mm_dpp} yields that $\pi^\star$ achieves the supremum in the multi–marginal problem \eqref{eq:mm_primal}.  
Conversely, for any global optimizer $\pi\in\mathcal M(\mu_0,\dots,\mu_N)$, conditioning on $(S_t,S_{t+1})$ shows that its time–$t$ transition $\pi_t$ must maximize the one–step functional with cost $\widetilde c_t$. Hence Theorem~\ref{thm:fric_monotone} implies that $\pi_t$ has the stated left–monotone two–graph structure, mass identities, and equal–slope condition.

Finally, if the one–step optimizer at each time $t$ is unique (e.g., under strict convexity in $v$ and the rectangle MSM condition), then the corresponding $(T_d^{(t)},T_u^{(t)},\theta_t)$ are unique $\mu_t$–a.e., and every optimal $\pi$ must have the same time–$t$ transition $K_t^\star$.
\end{proof}

\section{Proof of theorem~\ref{thm:strong_duality}}
\label{sec:proof_strong_duality}

We first recall the primal problem for better readability
\begin{equation*}
\mathbf{V} 
= \sup_{\pi \in \mathcal{M}(\mu_0,\dots,\mu_N)} 
\left\{ \mathbb{E}_\pi[\Phi] 
- \sum_{t=0}^{N-1} \int_{\mathbb{R}^2} f^{(a,b)}_t(x, y - x) \, \mathrm{d}\pi_t(x, y)
\right\}.
\end{equation*}

In this proof we combine the convex duality with martingale compactness arguments following 
\cite{BLP2013,HLT2016,BeiglbockNutzTouzi2017,DolinskySoner2014,galichon2014stochastic}. 
Throughout, \(t\in\{0,\dots,N-1\}\). 
The marginals \((\mu_t)_{t=0}^N\) satisfy the convex ordering of Assumption~\ref{ass:convex_order}, 
the payoff \(\Phi\) is admissible as per \cref{def:admissible_payoff}, 
and each friction cost \(f_t^{(a,b)}:\R\times\R\to[0,\infty]\) satisfies \cref{ass:fric-assump}.

Define continuation values backward by
\[
\mathsf V_N=\Phi,\qquad
\mathsf V_t(x):=\sup_{\pi\in\mathcal M(\mu_t,\dots,\mu_N)}
\E_\pi\left[\Phi(S_0,\dots,S_N)\,\middle|\,S_t=x\right].
\]
The one time–step adjusted cost reads
\[
\widetilde c_t(x,y)
=\mathsf V_{t+1}(y)-\mathsf V_{t+1}(x)-f_t^{(a,b)}(x,y-x).
\]
The superlinearity of \(f_t^{(a,b)}\) guarantees coercivity and compactness of the feasible set for each one–step problem
\[
\sup_{\pi_t\in\Pi^M(\mu_t,\mu_{t+1})}\int \widetilde c_t(x,y)\,\mathrm d\pi_t(x,y),
\]
which thus admits a maximizer, see \cite[Thm.~1.1]{BLP2013} for tightness and closedness of martingale couplings and \cite{HLT2016} for the one–step geometric structure.  
Backward induction along \(t=N-1,\dots,0\) then yields a dynamic programming identity identical in form to \eqref{eq:mm_dpp}.  
Telescoping the terms \(\mathsf V_{t+1}(y)-\mathsf V_{t+1}(x)\) across time shows that the global primal value \(\mathbf V\) equals the sum of the one–step values.

Next, we fix \(t\) and denote \(f_t:=f_t^{(a,b)}\).  
Consider the convex conjugate in the \(v\)-variable:
\[
f_t^\ast(x,h)
:=\sup_{v\in\R}\{h\,v - f_t(x,v)\},\qquad (x,h)\in\R^2.
\]
Superlinearity of \(f_t\) implies that \(f_t^\ast(x,\cdot)\) is finite everywhere, convex, and superlinear as \(|h|\to\infty\).  
The one time-step Lagrangian dual (see \cref{subsec:KKT}) introduces test functions \((\phi,\psi,h)\) satisfying
\[
\phi(x)+\psi(y)+h(x)(y-x)\le f_t(x,y-x)\qquad\forall x,y,
\]
which encodes both frictional and martingale constraints.  
By the Fenchel–Young inequality,
\[
h(x)(y-x)-f_t(x,y-x)\le f_t^\ast(x,h(x)),
\]
so the pointwise constraint is equivalently
\[
\phi(x)+\psi(y)\le f_t^\ast(x,h(x))\qquad\forall x,y.
\]
This yields the dual functional
\begin{equation}\label{eq:dual_one_step}
\inf_{(\phi,\psi,h)}\{\mu_t(\phi)+\mu_{t+1}(\psi):\ \phi(x)+\psi(y)\le f_t^\ast(x,h(x))\ \forall x,y\}.
\end{equation}

Eliminating \(\psi\) gives \(\psi(y)\le\inf_x(f_t^\ast(x,h(x))-\phi(x))\), hence
\begin{equation}\label{eq:dual_one_step_reduced}
\mu_t(\phi)+\mu_{t+1}(\psi)
=\mu_t(\phi)-\mu_{t+1}(\phi^\sharp_h),
\qquad
\phi^\sharp_h(y):=\sup_x\{\phi(x)-f_t^\ast(x,h(x))\}.
\end{equation}
Applying Fenchel–Rockafellar duality on the Banach space \(L^1(\mu_t)\times L^1(\mu_{t+1})\), the value of \eqref{eq:dual_one_step} equals the primal
\[
\sup_{\pi_t\in\Pi^M(\mu_t,\mu_{t+1})}\int \widetilde c_t(x,y)\,\mathrm d\pi_t(x,y).
\]
This is the one time–step frictional martingale duality, that extends the frictionless case of 
\cite{HLT2016,BeiglbockNutzTouzi2017} to state–dependent convex frictions.  
Superlinearity ensures finiteness and existence of minimizers by coercivity in \(h\) and weak compactness of the feasible set for \((\phi,\psi)\).

Summing \eqref{eq:dual_one_step_reduced} over \(t=0,\dots,N-1\) and undoing the dual shift by the continuation \(\mathsf V_{t+1}\) (cf.\ Lemma~\ref{lem:dual_shift_KKT}) yields
\begin{equation}\label{eq:global_dual_final}
\mathbf D
=\inf_{\substack{(\varphi_t,\psi_t,h_t)_{t=0}^{N-1}\\
\varphi_t(x)+\psi_t(y)+h_t(x)(y-x)\le f_t(x,y-x)}}
\sum_{t=0}^{N-1}\bigl(\mu_t(\varphi_t)+\mu_{t+1}(\psi_t)\bigr)
+\sum_{t=0}^{N-1}\bigl(\mu_{t+1}(\mathsf V_{t+1})-\mu_t(\mathsf V_{t+1})\bigr).
\end{equation}
The last sum telescopes to \(-\mu_0(\mathsf V_1)+\mu_N(\Phi)\), 
absorbing the former into \(\varphi_0\) recovers the semi–static (predictable) dual familiar from 
\cite{BLP2013,galichon2014stochastic}, 
now augmented by state-dependent frictions.  
Replacing \(\psi_t\) via the constraint gives an equivalent predictable increment formulation involving \(f_t^\ast(x,h_t(x))\). One can refer to \cite{BeiglbockNutzTouzi2017} for the analogous frictionless reduction.  
Summing the one–step dual equalities across \(t\) then yields \(\mathbf V=\mathbf D\).

\smallskip
Finally, we discuss the existence of optimizers for both the primal and dual problems. Let $\{\pi^n\}\subset\mathcal M(\mu_0,\dots,\mu_N)$ be maximizing for \eqref{eq:primal_problem_revised}.  
Convex order and bounded \(p\)-moments (\(p\ge2\)) imply tightness.  
Superlinearity of \(f_t^{(a,b)}\) ensures uniform integrability of increments \(S_{t+1}-S_t\) via de la Vallée–Poussin.  
By Prokhorov, there exists a subsequence with \(\pi^n\Rightarrow\bar\pi\in\mathcal M(\mu_0,\dots,\mu_N)\). 
Martingale closedness follows as in \cite[Sec.~2]{BLP2013}.  
Continuity of \(\pi\mapsto\int\Phi\,\mathrm d\pi\) (admissibility of \(\Phi\)) and lower semicontinuity of 
\(\pi\mapsto\int f_t^{(a,b)}\,\mathrm d\pi\) (convexity and superlinearity) imply upper semicontinuity of the objective, 
so \(\bar\pi\) attains \(\mathbf V\).

Let \((\varphi_t^n,\psi_t^n,h_t^n)\) be minimizing for \eqref{eq:global_dual_final}.  
Superlinearity of \(f_t^\ast\) yields coercivity in \(h_t^n\), giving uniform integrability in \(L^1(\mu_t)\).  
The feasible set is closed under \(L^1\)–limits by Fatou’s lemma and the pointwise inequality.  
By a Komlós subsequence argument, and after normalizing \((\varphi_t^n,\psi_t^n)\) by additive constants, 
we obtain a limit feasible triple \((\varphi_t,\psi_t,h_t)\) attaining \(\mathbf D\), 
see \cite[Sec.~3]{BLP2013} for the closedness of the superhedging cone argument.

\smallskip
Combining the one time–step Fenchel–Rockafellar duality \eqref{eq:dual_one_step}--\eqref{eq:dual_one_step_reduced}, telescoping structure \eqref{eq:global_dual_final}, and compactness arguments yields strong duality:
\[
\boxed{\ \mathbf V=\mathbf D,\quad 
\text{and both extrema are attained.}\ }
\]
\qed

\appendix


\section*{Appendix~A. Standard technical results}
\addcontentsline{toc}{section}{Appendix~A. Standard technical results}

This appendix collects several analytic tools that have been borrowed from convex analysis and probability theory.
Their proofs can be found in \cite{Rockafellar1970,Dudley2002,Bogachev2007}.

\begin{lemma}[Fatou’s lemma]
\label{lem:fatou}
Let $(\Omega,\mathcal F,\mu)$ be a measure space and $(f_n)_{n\ge1}$ a sequence of nonnegative measurable functions. Then
\[
\int_\Omega \liminf_{n\to\infty} f_n\,\mathrm d\mu
\;\le\;
\liminf_{n\to\infty}\int_\Omega f_n\,\mathrm d\mu.
\]
If $(f_n)$ is uniformly integrable and $f_n\to f$ in measure, then
$\int f_n\,\mathrm d\mu\to\int f\,\mathrm d\mu$.
\end{lemma}

\begin{theorem}[Fenchel–Rockafellar duality]
\label{thm:fenchel_rockafellar}
Let \(X\) be a Banach space, \(X^\ast\) its dual, and \(F,G:X\to(-\infty,+\infty]\) proper, convex, lower semicontinuous.  
If \(0\in\operatorname{core}(\operatorname{dom}F-\operatorname{dom}G)\), then
\[
\inf_{x\in X}\{F(x)+G(x)\}
\;=\;
\sup_{x^\ast\in X^\ast}\{-F^\ast(-x^\ast)-G^\ast(x^\ast)\},
\]
and both extrema are attained.
\end{theorem}

\begin{theorem}[Kuratowski–Ryll-Nardzewski measurable selection theorem]
\label{thm:kuratowski_ryll_nardzewski}
Let $(\Omega,\mathcal F)$ be a measurable space and \(K:\Omega\rightrightarrows X\) a set-valued map into a Polish space \(X\) such that for every open \(U\subset X\),
\(\{\omega:K(\omega)\cap U\ne\emptyset\}\in\mathcal F\).
Then there exists a measurable selector \(s:\Omega\to X\) with \(s(\omega)\in K(\omega)\) for all \(\omega\in\Omega\).
\end{theorem}

\begin{theorem}[Rokhlin disintegration theorem]
\label{thm:rokhlin_disintegration}
Let $(\Omega_1,\mathcal F_1)$ and $(\Omega_2,\mathcal F_2)$ be standard Borel spaces,  
and let \(\pi\in\mathcal P(\Omega_1\times\Omega_2)\) with first marginal \(\mu:=\pi\circ\mathrm{pr}_1^{-1}\).  
Then there exists a measurable family of probability measures
\(\{\pi(\cdot\mid x)\}_{x\in\Omega_1}\subset\mathcal P(\Omega_2)\)
such that for all bounded measurable \(f:\Omega_1\times\Omega_2\to\R\),
\[
\int f(x,y)\,\pi(\mathrm d x,\mathrm d y)
=\int_{\Omega_1}\left(\int_{\Omega_2}f(x,y)\,\pi(\mathrm d y\mid x)\right)\mu(\mathrm d x).
\]
Moreover, the conditional law \(x\mapsto \pi(\cdot\mid x)\) is measurable in the sense of weak convergence.
\end{theorem}

\begin{theorem}[Ionescu–Tulcea extension theorem]
\label{thm:ionescu_tulcea}
Let \((E_t,\mathcal E_t)_{t\ge0}\) be measurable spaces and \(K_t:(E_0\times\cdots\times E_t)\to\mathcal P(E_{t+1})\) transition kernels.  
Given an initial measure \(\mu_0\in\mathcal P(E_0)\), there exists a unique probability measure  
\(\pi\in\mathcal P\big(\prod_{t\ge0}E_t\big)\) such that for every measurable \(A_{t+1}\subset E_{t+1}\),
\[
\pi(S_{t+1}\in A_{t+1}\mid S_0,\dots,S_t)=K_t(S_0,\dots,S_t;A_{t+1})\quad\text{a.s.}
\]
The marginals are defined recursively by
\(\mu_{t+1}(A)=\int K_t(x_0,\dots,x_t;A)\,\mu_t(\mathrm d x_0,\dots,\mathrm d x_t)\).
\end{theorem}

\begin{theorem}[Prokhorov compactness theorem]
\label{thm:prokhorov}
A family \(\mathcal P_0\subset\mathcal P(E)\) of probability measures on a Polish space \(E\)  
is relatively compact under the topology of weak convergence if and only if it is tight,  
i.e. for every \(\varepsilon>0\) there exists a compact \(K_\varepsilon\subset E\) such that
\(\inf_{\pi\in\mathcal P_0}\pi(K_\varepsilon)\ge1-\varepsilon\).
Every tight sequence admits a weakly convergent subsequence.
\end{theorem}

\begin{lemma}[de la Vallée–Poussin criterion for uniform integrability]
\label{lem:dlvp}
A family \(\mathcal F\subset L^1(\mu)\) is uniformly integrable if and only if  
there exists a convex, increasing function \(\Phi:\R_+\to\R_+\) with \(\Phi(x)/x\to\infty\) as \(x\to\infty\)
such that
\[
\sup_{f\in\mathcal F}\int \Phi(|f|)\,\mathrm d\mu<\infty.
\]
\end{lemma}

\begin{lemma}[Komlós subsequence theorem]
\label{lem:komlos}
Let $(X_n)_{n\ge1}$ be a sequence bounded in \(L^1(\mu)\) on a probability space \((\Omega,\mathcal F,\mu)\).  
Then there exists a subsequence $(X_{n_k})_{k\ge1}$ and \(X\in L^1(\mu)\) such that the Cesàro means converge a.s.:
\[
\frac{1}{m}\sum_{k=1}^m X_{n_k}\ \longrightarrow\ X\quad\text{a.s. as }m\to\infty.
\]
In particular, \(X_{n_k}\) converges to \(X\) weakly in \(L^1(\mu)\).
\end{lemma}


\bibliographystyle{apalike}
\bibliography{references}

\end{document}